\documentclass[oneside]{amsart}

\usepackage{amssymb}
\usepackage{amsmath}
\usepackage{amsfonts}
\usepackage{mathrsfs}
\usepackage{pifont}
\usepackage{amsthm}
\usepackage{hyperref}
\usepackage{amscd}
\usepackage{color}
\usepackage{enumerate}
\usepackage{tikz-cd}
\usetikzlibrary{calc}
\usepackage{subcaption}
\usepackage{caption}
\usepackage{float}
\usepackage{enumitem}
\usepackage{comment}

\usepackage{tabularx}
\usepackage{booktabs}
\usepackage{longtable}

\usepackage{makecell}
\usepackage{diagbox}
\usepackage{varioref}

\usepackage{nicematrix}
\usepackage{romannum}

\usepackage[linesnumbered,ruled,vlined]{algorithm2e}

\newcommand{\prim}{\mathrm{prim}}
\newcommand{\cross}{\times}

\newcommand{\ZZ}{\mathbb{Z}}
\newcommand{\KK}{\mathbb{K}}
\newcommand{\RR}{\mathbb{R}}

\newcommand{\diam}{\mathrm{diam}}

\newcommand{\bQ}{\mathbf{Q}}
\newcommand{\bR}{\mathbf{R}}
\newcommand{\bP}{\mathbf{P}}
\newcommand{\bg}{\mathbf{g}}
\newcommand{\ba}{\mathbf{a}}
\newcommand{\bz}{\mathbf{z}}
\newcommand{\bx}{\mathbf{x}}
\newcommand{\tbx}{\tilde{\bx}}
\newcommand{\tB}{\tilde{B}}
\newcommand{\mcF}{\mathcal{F}}
\newcommand{\mcA}{\mathcal{A}}
\newcommand{\mcV}{\mathcal{V}}
\newcommand{\mcW}{\mathcal{W}}
\newcommand{\mcC}{\mathcal{C}}
\newcommand{\mcG}{\mathcal{G}}

\DeclareMathOperator{\In}{In}

\DeclareMathOperator{\spec}{Spec}

\DeclareMathOperator{\id}{id}
\DeclareMathOperator{\Hom}{Hom}
\DeclareMathOperator{\Der}{Der}

\DeclareMathOperator{\len}{len}

\usepackage{tipa}

\usepackage{theoremref}
\newtheorem{theorem}[equation]{Theorem}
\newtheorem{proposition}[equation]{Proposition}
\newtheorem{lemma}[equation]{Lemma}
\newtheorem{corollary}[equation]{Corollary}

\theoremstyle{definition}

\newtheorem{definition}[equation]{Definition}

\newtheorem{example}[equation]{Example}
\newtheorem{remark}[equation]{Remark}

\numberwithin{equation}{subsection}

\usepackage{newfloat}
\DeclareFloatingEnvironment[
fileext=los,
listname={List of Diagrams},
name=Diagram,
placement=h,
within=none,
]{mdiagram}

\title{Gr\"obner Cones for Finite Type Cluster Algebras}

\author{Nathan Ilten}
\address{Department of Mathematics, Simon Fraser University,
	8888 University Drive, Burnaby BC V5A1S6, Canada}
\email{nilten@sfu.ca}

\author{Karolyn So}
\address{Department of Mathematics, Simon Fraser University,
	8888 University Drive, Burnaby BC V5A1S6, Canada}
\email{wsa57@sfu.ca}
\begin{document}
\pagenumbering{arabic}
\begin{abstract}
	Let $\mcA$ be a cluster algebra of finite cluster type. We study the Gr\"obner cone $\mcC_\mcA$ parametrizing term orders inducing an initial degeneration of the ideal $I_\mcA$ of relations among the cluster variables of $\mcA$ to the ideal generated by products of incompatible cluster variables.
We show that for any cluster variable $v$, the weight induced by taking compatibility degrees with $v$ belongs to $\mcC_\mcA$. This allows us to construct an explicit circular term order and prove a conjecture of Ilten, N\'ajera Ch\'avez, and Treffinger. Furthermore, we give explicit descriptions of the rays and lineality spaces of $\mcC_\mcA$ in terms of combinatorial models for cluster algebras of types $A_n$, $B_n$, $C_n$, $D_n$ with a special choice of frozen variables, and in the case of no frozen variables.
\end{abstract}
\maketitle

\section{Introduction}
Let $I_{2,n}$ be the homogeneous ideal of the Grassmannian $G(2,n)$ in its Pl\"ucker embedding; this ideal is generated by the quadrics
\[
	x_{ik}x_{jl}-x_{ij}x_{kl}-x_{il}x_{jk}
\]
for $1\leq i<j<k<l\leq n$. There exists a term order $\prec$ for which the initial ideal of $I_{2,n}$ with respect to $\prec$ is generated by the left-most monomials $x_{ik}x_{jl}$, see e.g.~\cite[pg.~131]{Sturmfels_AlgorithmsInInvariantTheory}. Sturmfels calls such term orders \emph{circular}. The resulting initial ideal $I_{2,n}^\cross$ is the Stanley-Reisner ideal of the dual associahedron and enjoys many wonderful properties, see e.g.~\cite{associahedron}.

For which term orders is the initial ideal of $I_{2,n}$ exactly $I_{2,n}^\cross$? For terms orders arising from a weight $\omega$, this answer is encoded by the so-called Gr\"obner cone $\mcC$ of $I_{2,n}$ with respect to $I_{2,n}^\cross$. Indeed, the cone $\mcC$  is the closure in Euclidean space of all weights $\omega$ for which the initial ideal of $I_{2,n}$ is $I_{2,n}^\cross$, see \S\ref{sec:groebner} for details.
Speyer and Sturmfels explicitly describe the rays and the lineality space of the cone $\mcC$ in \cite[\S 4]{tropgrass}. We will recover this result as a special case of our Theorems \ref{Proposition_ABCn_LinSp} and \ref{Theorem_Statement_Result3_ABCn}.

The above example of the Grassmannian $G(2,n)$ is the starting point of this paper. Introduced by Fomin and Zelevinsky \cite{Fomin_IFoundations}, a \emph{cluster algebra} is a kind of commutative algebra endowed with special combinatorial structure, see \S\ref{sec:cluster}. The homogeneous coordinate ring of the Grassmannian $G(2,n)$ is an example of a cluster algebra of type $A_{n-3}$.  
In this paper, we will study Gr\"obner cones related to cluster algebras.

As part of its combinatorial structure, a cluster algebra $\mcA$ has a distinguished set of generators, called \emph{cluster variables} and \emph{frozen variables}, giving rise to a presentation
\[
    0 \longrightarrow 
    I_{\mcA}\longrightarrow
    S \longrightarrow
    \mcA \longrightarrow
    0
\]
where $S$ is a polynomial ring.
The algebra $\mcA$ is of \emph{finite cluster type} when this set of generators is finite, making $S$ into a polynomial ring in finitely many variables. In analogy to the ideal $I_{2,n}^\cross$ above, we let $I_\mcA^\cross$ be the ideal of $S$ generated by products of the variables of $S$ corresponding to \emph{incompatible} generators of $\mcA$; see \S\ref{sec:groebner} for a precise description.
In the case of finite cluster type, Ilten, N\'ajera Ch\'avez, and Treffinger show that there is a term order on $S$ for which the initial ideal of $I_\mcA$ is exactly $I_\mcA^\cross$ \cite[Corollary 5.3.2]{ilten2021deformation}. As in the case of the Grassmannian, we will call such term orders \emph{circular}.

Let $\mcC_\mcA$ be the Gr\"obner cone parametrizing all weights giving circular terms orders. Inequalities implicitly describing $\mcC_\mcA$ may be found in loc.~cit.
The goal of this paper is to instead give an explicit description of the generators of, or at least some elements of, the Gr\"obner cone $\mcC_\mcA$.

In order to describe our results, we recall several more facts about cluster algebras of finite cluster type. Firstly, given any two cluster variables $x_1,x_2$ of $\mcA$, there is a non-negative integer $(x_1||x_2)$ called the \emph{compatibility degree} of $x_1$ and $x_2$, see \S\ref{sec:root} and \cite[\S 3.1]{Fomin_YSystemsAndGeneralizedAssociahedra}. Secondly, cluster algebras of finite cluster type may be classified according to \emph{type}, and these types are exactly the types of semisimple lie algebras, or equivalently, Cartan matrices of finite type \cite[Theorem 1.4]{Fomin_IIFiniteTypeClassification}.
Finally, cluster algebras of the classical types $A_n$, $B_n$, $C_n$, $D_n$ may be encoded via combinatorial models involving (pairs of) diagonals in regular polygons, see \S \ref{sec:comb} and \cite[\S12]{Fomin_IIFiniteTypeClassification}. These models come equipped with a special choice of frozen variables for the corresponding cluster algebras $\mcA$.

We now summarize our results. Let $\mcA$ be any cluster algebra of finite cluster type.
\begin{enumerate}
	\item For any cluster variable $v$ of $\mcA$, the tuple $\omega_v:=[(v||y)]_y$ as $y$ ranges over all cluster and frozen variables gives an element of the Gr\"obner cone $\mcC_\mcA$ (Theorem \ref{thm:compdegree}).\label{result:1}
	\item The sum of all weights $\omega_v$ gives rise to a circular term order (Corollary \ref{Corollary_ExplicitCircTermOrder}).\label{result:2}
	\item Let $\mcA$ be a cluster algebra with no frozen variables of type $A_n$, $B_n$, or $C_n$ for $n$ even, or of type $F_4$. We describe the generators of the rays of $\mcC_\mcA$ in terms of alternating sums of the $\omega_v$ (Theorem \ref{Therorem_Result2_Main}).\label{result:3}
	\item In types $A_n$, $B_n$, $C_n$, and $D_n$, we describe the lineality space and ray generators of $\mcC_\mcA$ in terms of the corresponding combinatorial model in the case of special frozen variables (Propositions \ref{Proposition_ABCn_LinSp} \ref{Proposition_Result3_Dn_Linsp} and Theorems \ref{Theorem_Statement_Result3_ABCn} \ref{Theorem_Statement_Result3_Dn}) and the case of no frozen variables (Theorems \ref{thm:cf1}, \ref{thm:cf2}, \ref{thm:cf3}, \ref{thm:cf4}).\label{result:4}
	\item We give a positive answer to Conjecture 6.2.9 of \cite{ilten2021deformation}, which states that the multidegrees of first order embedded deformations of $\spec (S/I_\mcA^\cross)$ that become trivial when forgetting the embedding do not interact with the semigroup generated by multidegrees of first order deformations induced by the universal cluster algebra of $\mcA$ (Corollary \ref{Corollary_ModCompDeg}).\label{result:5}
\end{enumerate}

\noindent Our approach to the above results in the classical types $A_n$, $B_n$, $C_n$, and $D_n$  is via the above-mentioned combinatorial models. For the exceptional cases we make use of a computer calculation.

In related work, Bossinger has studied tropicalizations of cluster varieties \cite{bossinger}. For finite type cluster algebras of \emph{full rank}, this gives a method for describing the rays and lineality space of $\mcC_\mcA$ in terms of the $\bg$-vectors of $\mcA$ \cite[Theorem 1.1 and Corollary 3.19]{bossinger}. It would be interesting to compare this with our explicit descriptions of $\mcC$ for the $A_n,B_n,C_n$ cases with special frozen variables, or no frozen variables and $n$ even. (The remaining cases we consider do not have full rank). Similarly, it would be interesting to see if Bossinger's results can be used to provide a uniform cluster-theoretic proof of Theorem \ref{thm:compdegree}.

The remainder of this paper is organized as follows. In \S\ref{sec:prelim} we cover preliminaries: cluster algebras are introduced in \S\ref{sec:cluster}, we discuss the Gr\"obner cone $\mcC_\mcA$ in \S\ref{sec:groebner}, and \S\ref{sec:root} introduces compatibility degree and discusses the connection with root systems. In \S\ref{sec:comb} we present the combinatorial models we will be using to study the classical types.
In \ref{sec:TO} we study the relationship between $\mcC_\mcA$ and compatibility degree, obtaining the results \eqref{result:1}, \eqref{result:2}, \eqref{result:3}, and \eqref{result:5} from above. The result \eqref{result:4} is split between \S\ref{sec:frozen}, where we study the case with special frozen variables, and \S\ref{sec:nofrozen}, where we study the case with no frozen variables. The latter is obtained as a consequence of the former.

\subsection*{Acknowledgements} Both authors were partially supported by an NSERC Discovery grant. Part of the contents of this paper appeared as the MSc thesis of the second author.
\section{Preliminaries}\label{sec:prelim}
\subsection{Cluster basics}\label{sec:cluster}
We introduce basic notions from the theory of cluster algebras, see \cite[\S 3]{fomin2021introduction1to3} for details.
Throughout the paper, $\KK$ will denote a field of characteristic zero.

We fix two positive integers $n\leq m$, and a field $\mcF$ isomorphic to a field of rational functions in $m$ variables over $\KK$. 
Recall that an $n\times n$ integer matrix $B$ is \emph{skew-symmetrizable} if there exists a diagonal matrix $D$ with positive integer diagonal entries such that $DB$ is skew-symmetric.
An $n\times n$ integer matrix $B=(b_{ij})$ is \emph{indecomposable} if there is no proper subset $I\subsetneq \{1,\ldots,n\}$ such that $b_{ij}=0$ for all $i\in I$ and $j\notin I$.

\begin{definition}[Seed]
	A \emph{labeled seed of geometric type} (or simply a \emph{seed}) in $\mcF$ is a pair $(\tbx, \tB)$, where
    \begin{enumerate}
        \item 
        $\tbx=(x_1,\ldots, x_m)$ is an $m$-tuple of elements of $\mcF$ with $\mcF=\KK(x_1,\ldots, x_m)$;
        \item 
        $\tB=(b_{ij})$ is an $m\times n$ integer matrix, in which the top $n\times n$ submatrix $B$ is skew-symmetrizable.
    \end{enumerate}
\end{definition}

We call $\tB$ the \emph{extended exchange matrix} of the seed and its top $n\times n$ submatrix $B$  the \emph{exchange matrix}. The $n$-tuple 
$\bx=(x_1,\ldots, x_n)$ is the \emph{cluster} of this seed and its elements are called \emph{cluster variables}; $\tbx$ is called the \emph{extended cluster} and $x_{n+1},\ldots,x_m$ are called \emph{frozen variables}.

Given a seed, we can produce new seeds from it by a process called mutation:
\begin{definition}[Mutation]
    Let $(\tbx, \tB)$ be a seed. Fix an index $k\in\{1,\ldots, n\}$. The \emph{mutation} $\mu_k$ in direction $k$ transforms $(\tbx, \tB)$ to the new seed $\mu_k(\tbx, \tB)=(\tbx', \tB')$ defined as follows:
    \begin{enumerate}
        \item
        The extended exchange matrix $\tB'=(b_{ij}')$ is given by
        \begin{equation*}%\label{EqtMatrixMutation}
        b_{ij}'=\begin{cases}
            -b_{ij}  & \text{ if } i=k \text{ or } j=k , \\
            b_{ij} + \text{sgn}(b_{ik})\left[ b_{ik}b_{kj} \right]_+&\text{ otherwise,}
        \end{cases}
        \end{equation*}
        
        where $\left[a\right]_{+}:=\text{max}\{a, 0\}$.
        
        \item 
        The extended cluster $\tbx'=(x_1',\ldots, x_m')$ is given by $x_j'=x_j$ for $j\neq k$, and $x_k'\in\mathcal{F}$ is determined by the following equation called an \emph{exchange relation}
        \begin{equation}\label{EqtExchangeRelation}        x_k'x_k=\displaystyle\prod_{b_{ik}>0}x_i^{b_{ik}}+\prod_{b_{ik}<0}x_i^{-b_{ik}}.
        \end{equation}
        Here we use the convention that the empty product is equal to $1$.
    \end{enumerate}
\end{definition}
The mutation of a seed is again a seed.    
The variables on the left-hand side of Equation \eqref{EqtExchangeRelation} are said to be  \emph{exchangeable}.
    An exchange relation is said to be \emph{primitive} if at least one of the monomials on the right-hand side of Equation \eqref{EqtExchangeRelation} only consists of frozen variables. In this case, the other monomial on the right-hand side is called a \emph{primitive term}.
\begin{definition}[Cluster algebra]\label{DefClusterAlgebra}
	Let $(\tbx,\tB)$ be a seed. The \emph{cluster algebra} 
	$\mcA(\tbx,\tB)$ is the subalgebra of $\mcF$ generated by all cluster and frozen variables appearing in seeds obtained from $(\tbx,\tB)$ by a sequence of mutations.
\end{definition}

Up to isomorphism, the cluster algebra $\mcA(\tbx,\tB)$ only depends on the extended exchange matrix $\tB$ and not on $\tbx$, so we will frequently abbreviate $\mcA(\tbx,\tB)$ by $\mcA(\tB)$. The seed $(\tbx,\tB)$ is called the \emph{initial seed} of $\mcA(\tbx,\tB)$. Any seed $(\tbx',\tB')$ obtained from 
$(\tbx,\tB)$ via a sequence of mutations is called a seed of $\mcA(\tbx,\tB)$, and we note that $\mcA(\tbx',\tB')=\mcA(\tbx,\tB)$

A cluster algebra $\mcA$ has \emph{no frozen variables} if $n=m$. It is said to be of \emph{full rank} if for any (or equivalently for every) seed $(\tbx,\tB)$, $\tB$ has full rank.
The \emph{cluster variables} of $\mcA$ are the union of all cluster variables appearing in seeds of $\mcA$. Any cluster or frozen variables appearing in a common extended cluster of $\mcA$ are said to be \emph{compatible}. A cluster algebra $\mcA$ is said to be of \emph{finite cluster type} if it only has a finite number of cluster variables.
In this paper, we will only concern ourselves with cluster algebras of finite cluster type.

\begin{remark}
	The cluster algebras we have defined in Definition \ref{DefClusterAlgebra} are referred to in the literature as \emph{cluster algebras of geometric type}, see e.g.~\cite[Chapter 3]{fomin2021introduction1to3}. We follow the convention of loc.~cit.~and \cite{ilten2021deformation} of \emph{not} inverting the frozen variables. This has the consequence that the finite cluster type cluster algebras we consider with frozen variables are in general not complete intersection rings. 
\end{remark}

Given an $n\times n$ skew-symmetrizable matrix $B=(b_{ij})$, its \emph{Cartan counterpart} is the matrix $A=A(B)=(a_{ij})_{n\times n}$, where
    \[a_{ij}=\begin{cases}
        2 & \text{ if } i=j,\\
        -|b_{ij}| &\text{ if } i\neq j.
    \end{cases}\]
    The Cartan matrix $A(B)$ is indecomposable if and only if $B$ is. An indecomposable Cartan matrix is of \emph{finite type} if all principal minors are positive. More generally, a Cartan matrix $A(B)$ is of finite type if after conjugating by a permutation matrix, $A(B)$ is block diagonal with blocks indecomposable Cartan matrices of finite type. We recall that indecomposable finite type Cartan matrices are classified by Dynkin diagrams; this gives rise to the notion of the \emph{type} of a Cartan matrix of finite type.

\begin{theorem}[{\cite[Theorem 1.4]{Fomin_IIFiniteTypeClassification}}]\label{TheoremFiniteTypeClassification}
A cluster algebra $\mcA$ is of finite cluster type if and only if it admits a seed $(\tbx,\tB)$ such that the Cartan matrix $A(B)$ is of finite type.
\end{theorem}

Given a cluster algebra $\mcA$ of finite cluster type, its \emph{type} is the type of the Cartan matrix $A(B)$, where $B$ is the exchange matrix for any seed such that $A(B)$ is of finite type. It follows from \cite[Theorem 1.4]{Fomin_IIFiniteTypeClassification} that this is well defined.

\subsection{The Gr\"obner Cone}\label{sec:groebner}
Let $\mcA$ be a cluster algebra of finite cluster type. We will denote its set of cluster variables by $\mcV$ and its set of frozen variables by $\mcW$, and let 
\[
	\KK[\bz]=\KK\left[z_v ~|~ v\in \mcV\cup \mcW\right].
\]
We have a presentation
\[
    0 \longrightarrow 
    I_{\mcA}\longrightarrow
    \KK[\bz]\longrightarrow
    \mcA \longrightarrow
    0
\]
by sending $z_v$ to the corresponding cluster or frozen variable. 
In other words, we have that the cluster algebra $\mcA$ is isomorphic to $\KK[\bz]/I_{\mcA}$.
The ideal $I_\mcA$ includes polynomials corresponding to the exchange relations, but in general may contain additional relations.

\begin{definition}\label{defn:Icross}
	We let $I_\mcA^\cross$ be the ideal of $\KK[\bz]$ generated by products $z_v\cdot z_w$ where $v,w\in\mcV$ do not belong to a common cluster.
\end{definition}
\begin{remark}
	The ideal $I_\mcA^\cross$ may be viewed as the Stanley-Reisner ideal associated to the so-called cluster complex of the cluster algebra $\mcA$, see e.g.~\cite[\S2]{ilten2021deformation} for details and a precise statement.
\end{remark}

We denote by $\RR^{\mcV\cup \mcW}$ (or $\ZZ^{\mcV\cup \mcW}$) tuples of elements of $\RR$ (or $\ZZ$) indexed by $\mcV\cup \mcW$.
We will use multi-index notation for monomials of $\KK[\bz]$: for $\gamma\in \ZZ_{\geq 0}^{\mcV\cup \mcW}$, $\bz^\gamma:=\prod_v z_v^{\gamma_v}\in \KK[\bz]$.

Consider $\omega = (\omega_v)_{v\in\mcV\cup\mcW}\in \RR^{\mcV\cup\mcW}$. Recall that for any polynomial $f=\sum c_i \bz^{\ba_i}\in \KK[\bz]$, the \emph{initial form} $\In_\omega(f)$ of $f$ with respect to $\omega$ is the sum of all terms $c_i \bz^{\ba_i}$ such that the dot product $\omega \cdot \ba_i$ is maximal.
For an ideal $I$, the \emph{initial ideal} with respect to $\omega$ is
    \[\In_\omega(I):=\langle \In_\omega(f): f\in I\rangle.\]

In \cite[Corollary 5.3.2]{ilten2021deformation} it is shown that $I_\mcA^\cross$ is an initial ideal of $I_\mcA$. It thus makes sense to ask which weights give rise to this initial ideal.
\begin{definition}
	The \emph{Gr\"obner cone} $\mcC_\mcA$ of $I_\mcA$ with respect to the initial ideal $I_\mcA^\cross$ is the closure of the set of all $\omega\in \RR^{\mcV\cup\mcW}$ such that $\In_\omega(I_\mcA)=I_\mcA^\cross$.
\end{definition}

The goal of this paper is to give \emph{explicit} descriptions of some (or all) elements of $\mcC_\mcA$. Our starting point is the following \emph{implicit} description of $\mcC_\mcA$. Let $\deg:\mcV\cup \mcW\to \ZZ^{\mcV\cup\mcW}$ be the map sending $v$ to the standard basis vector $e_v$. This extends in the obvious way to a map on monomials in the cluster and frozen variables.

\begin{lemma}[cf. {\cite[Corollary 5.3.2]{ilten2021deformation}}]\label{Lemma_Of_The_Whole_Thesis}
	Let $\mcC_\prim$ be the cone in $\RR^{\mcV\cup \mcW}$ generated by 
 \[\deg(xx')-\deg(y_1)\in \ZZ^{\mathcal{V}\cup \mathcal{W}}\]
        where $x, x',y_1$ appear in a primitive exchange relation $xx'=y_1+y_2$, and $y_1$ is a primitive term in this relation.
Then $\mcC_\mcA\subseteq \RR^{\mcV\cup \mcW}$ is the cone dual to $\mcC_\prim$.
\end{lemma}

We call the generators of $\mathcal{C}_{\mathrm{prim}}$ \emph{degrees induced by primitive exchange relations}. 
If $\mathcal{A}$ is a cluster algebra with indecomposable exchange matrix and not of type $A_1$, every primitive exchange relation $P$ has a unique primitive term. In these cases, we denote the induced degree $d_P$.
For type $A_1$, there is only one primitive exchange relation, and it has two primitive terms. We will deal with this case separately.

\subsection{Compatibility Degree and Root Systems}\label{sec:root}
Let $\mcA$ be a cluster algebra of finite cluster type. 
In \cite[Theorem 1.9]{Fomin_IIFiniteTypeClassification}, Fomin and Zelevinsky show that the cluster variables of $\mcA$ can be identified with so-called almost positive roots of a root system $\Phi$. Using this, in \cite[\S 3.1]{Fomin_YSystemsAndGeneralizedAssociahedra} they define a function from pairs of cluster variables to $\ZZ_{\geq 0}$ called the \emph{compatibility degree}. For cluster variables $x_1,x_2\in \mcA$ we denote their compatibility degree by $(x_1||x_2)$.
This has the following significance:
\begin{enumerate}
        \item $(x_1||x_2)=0$ if and only if $x_1$ and $x_2$ are compatible; and
        \item $(x_1||x_2)=(x_2||x_1)=1$ if and only if $x_1$ and $x_2$ are exchangeable.
\end{enumerate}
By convention, we extend the definition of compatibility degree to include frozen variables of $\mathcal{A}$ by setting $(x ||y)=0$ if either of $x,y$ is a frozen variable.

For a precise definition of compatibility degree, we refer the reader to \cite[\S 3.1]{Fomin_YSystemsAndGeneralizedAssociahedra}. However, in this paper we will use two interpretations of compatibility degree. Firstly, for cluster algebras of types $A_n$, $B_n$, $C_n$, or $D_n$, compatibility degree may be interpreted in terms of the combinatorial models that we recall in \S\ref{sec:comb}. Secondly, there is an interpretation of compatibility degree in terms of operations on the weight lattice of a root system that we now recall, see \cite{Yang_ClusterAlgebrasOfFiniteTypeViaCoxeterElementsAndPrincipalMinors}. We will use this for the exceptional types.

Let $\Phi$ be any root system spanning a vector space $V$ with bilinear pairing $\langle\cdot,\cdot \rangle$. 
See e.g.~\cite[\S1]{Humphreys_ReflectionGroupsAndCoxeterGroups} and\cite[\S8]{Hall_LieGroupsLieAlgebrasAndRepresentations} for details on root systems.
Fix $\alpha_1,\ldots,\alpha_n$ a choice of simple roots and let $s_1,\ldots,s_n$ be the corresponding simple reflections; these generate the Weyl group $W$.

The \emph{fundamental weights} are the unique elements $\omega_1,\ldots,  \omega_n\in V$ such that 
    \[
    2\frac{\langle\omega_j, \alpha_k\rangle}{\langle\alpha_k, \alpha_k\rangle}=\delta_{jk}
    \]
    for $j,k=1,\ldots, n$.
    The \emph{weight lattice} $\Lambda$ is the $\ZZ$-linear span of the fundamental weights.
Recall that the \emph{Cartan matrix} of $\Phi$ is the integer square matrix $A=(a_{ij})_{n\times n}$, where \[a_{ij}=\frac{2\langle\alpha_i,\alpha_j\rangle}{\langle\alpha_i,\alpha_i\rangle}\] for all $1\leq i,j \leq n $. 

An arbitrary element $w\in W$ can be written as a product $w=s_{i_1}s_{i_2}\ldots s_{i_r}$.  
When $r$ in the expression above is minimized, we call $r$ the \emph{length} of $w$ and say that this expression is a \emph{reduced expression} of $w$.
Let $c$ be a \emph{Coxeter element} in $W$, that is, $c$ is a product of all $n$ simple reflections.
Denote by $\zeta$ the unique longest element in $W$ and denote by $h(i;c)$ the minimum positive integer such that $c^{h(i;c)}\omega_i=\zeta\omega_i$. 
By \cite[Proposition 1.3]{Yang_ClusterAlgebrasOfFiniteTypeViaCoxeterElementsAndPrincipalMinors}, $h(i;c)$ is a finite number.

\begin{remark}
    Note that $\zeta \omega_i=-\omega_j$ for some $j$ \cite[Chapter IV, \S1.6, Corollary 3]{Bourbaki}, so the definition of $h(i;c)$ above implies that $c^{h(i;c)}\omega_i=-\omega_j$.
\end{remark}

With this setting, we construct a cluster algebra by defining the initial exchange matrix. Let $B_c=(b_{ij})_{i,j=1,\ldots n}$ be the integer matrix where 
\[
b_{ij}=\begin{cases}
    -a_{ij} & \text{ if } i \prec_c j,\\
    a_{ij}  & \text{ if } j \prec_c i,\\
    0       & \text{ otherwise, }
\end{cases}
\]
where $i \prec_c j$ if and only if $s_i$ precedes $s_j$ in all reduced expressions of $c$. Here, the $a_{ij}$ are the entries of the Cartan matrix of $\Phi$.
By \cite[\S 1]{Yang_ClusterAlgebrasOfFiniteTypeViaCoxeterElementsAndPrincipalMinors}, $B_c$ is skew-symmetrizable.
Let $\mcA(B_c)$ be the cluster algebra defined by the matrix $B_c$ (with no frozen variables).
By \cite[Theorem 1.2]{Yang_ClusterAlgebrasOfFiniteTypeViaCoxeterElementsAndPrincipalMinors}, the cluster algebra $\mcA(B_c)$ is of finite cluster type, and of the same type as the root system $\Phi$. 

To any cluster variable $x\in\mcA(B_c)$, one may associate a so-called $\bg$-vector $\bg(x)\in \ZZ^{n}$. The precise definition of $\bg$-vectors is not needed here;  interested readers may consult \cite[\S6]{Fomin_IVCoefficients} and \cite{Yang_ClusterAlgebrasOfFiniteTypeViaCoxeterElementsAndPrincipalMinors}.

\begin{theorem}[{\cite[Theorem 1.4]{Yang_ClusterAlgebrasOfFiniteTypeViaCoxeterElementsAndPrincipalMinors}}]\label{TheoremBijectionGVectors}
    The cluster variables of $\mcA(B_c)$ are in bijection with the elements of the set
    \[
    \Pi(c):=\{c^k\omega_i : 1\leq i\leq n,\ 0\leq k \leq h(i;c)\},
    \]
    by the map
    \[
    \psi: x \mapsto \sum_{i}\mathbf{g}(x)_i\omega_i.
    \]
\end{theorem}
\noindent The above bijection extends to a bijection between cluster monomials of $\mcA(B_c)$, that is,  monomials consisting of variables from a common cluster, and the weight lattice $\Lambda$, see \cite[\S1]{Stella_ExchangeRElationsForFiniteTypeClusterAlgebraWithAcyclicInitialSeedAndPrincipalCoefficients}.
Indeed, any cluster monomial $y_1^{\gamma_1}\cdots y_n^{\gamma_n}$ maps to $\sum_i \gamma_i\psi(y_i)\in\Lambda$.
For any weight $\lambda\in \Lambda$, we denote by $x_\lambda$ be the corresponding cluster monomial of $\mcA(B(c))$.

The set $\Pi(c)$ is endowed with a permutation $\tau_c$ defined by 
\begin{equation*}%\label{Equation_def_tau_c}
\tau_c(\lambda):=\begin{cases}
    \omega_i & \text{ if } \lambda=-\omega_i,\\
    c\lambda & \text{ otherwise,}    
\end{cases}    
\end{equation*}
which extends to a piecewise linear map on the whole of $\Lambda$.

\begin{theorem}[{\cite[Proposition 5.1]{Yang_ClusterAlgebrasOfFiniteTypeViaCoxeterElementsAndPrincipalMinors}}]\label{Theorem_c-compdeg_def}
    There is a unique $\tau_c$-invariant function $(\cdot|| \cdot)_c$ on pairs of elements of $\Pi(c)$ satisfying the initial conditions
    \begin{enumerate}[label=(\arabic*)]
        \item 
		$(\omega_i || \omega_j)_c=0$ for $i,j\in\{1,\ldots,n\}$,
        \item 
		$(\omega_i || \lambda)_c = [(c^{-1}-\id)\lambda : \alpha_i]_+$ for $\lambda \in \Pi(c)\backslash
	\{\omega_j:j\in \{1,\ldots,n\}\}$,
    \end{enumerate}
    where $[v:\alpha_i]_+$ denotes the maximum of $0$ and the $i$-th coefficient of $v$ when expressed as a linear combination of simple roots.
    This function is called the $c$-compatibility degree. 
    Furthermore, under the bijection $\psi$ of Theorem \ref{TheoremBijectionGVectors}, for cluster variables $x,y\in \mcA(B_c)$ we have
    \[
    (x || y)= (\psi(x) || \psi(y))_c.
    \]
\end{theorem}

We may use Theorem \ref{Theorem_c-compdeg_def} to compute compatibility degree as follows.
By the definition of $\Pi(c)$, for every weight $\lambda\in \Pi(c)$, 
there is some natural number $k$ such that $c^k \lambda = \zeta \omega_j=-\omega_i$ for some $i$ and $j$.
Taking the minimal such $k$ and applying $\tau_c$ one more time, we have $\tau_c^{k+1}(\lambda)=\omega_i$. 
Using this fact, we may compute for any $\mu\in\Pi(c)$ that 
\begin{align*}
    (\lambda||\mu)_c    &= (\tau_c^{k+1}(\lambda)||\tau_c^{k+1}(\mu))_c\\
                        &= (\omega_i||\tau_c^{k+1}(\mu))_c\\
                        &= 
                        \begin{cases}
				0 & \text{ if } \tau_c^{k+1}(\mu)=\omega_j \text{ for some }j\in\{1,\ldots,n\},\\
                            \left[(c^{-1}-\id)(\tau_c^{k+1}(\mu)) : \alpha_i\right]_+ & \text{ otherwise.}
                        \end{cases}
\end{align*}

To prove Theorem \ref{thm:compdegree} we will also need good control over the exchange relations in $\mcA(B(c))$.
We say that weights $\lambda,\mu\in\Pi(c)$ are \emph{exchangeable} if the corresponding cluster variables in $\mcA(B(c))$ are exchangeable.
\begin{theorem}[{\cite[Theorem 1.5]{Yang_ClusterAlgebrasOfFiniteTypeViaCoxeterElementsAndPrincipalMinors}}, {\cite[Corollary 4.2]{Stella_Polyhedral_Models}}]\label{Theorem_GenPrimExchRel}
    For any two exchangeable weights $\lambda$ and $\mu$ in $\Pi(c)$, the set
    \begin{equation*}%\label{EquationSetOfExchangeableWeights}
    \{\tau_c^{-k}(\tau_c^{k}(\lambda)+\tau_c^{k}(\mu))\}_{k\in\ZZ}    
    \end{equation*}
    consists of exactly two weights. One of them is $\lambda +\mu$; denote the other by $\lambda \uplus_c \mu$. 
    The exchange relation $\mathcal{A}(c)$  between $x_\lambda$ and $x_\mu$ is exactly
    \begin{equation*}%\label{EquationPrimExchWeights}
    x_\lambda x_\mu = x_{\lambda+\mu} + x_{\lambda \uplus_c\mu}.  
    \end{equation*}
Furthermore, a relation is primitive if and only if $\lambda$ and $\mu$ differ by an application of $\tau_c$.
\end{theorem}

\section{Combinatorial Models}\label{sec:comb}
\subsection{Notation}
In this section, we will recall combinatorial models for certain cluster algebras of types $A_n$, $B_n$, $C_n$, and $D_n$ related to triangulations of regular polygons. These cluster algebras come equipped with a choice of frozen variables; to recover the case with no frozen variables, one can simply set all frozen variables to be $1$.

Before diving into the model for each type, we define some universal notation about regular polygons.
Let $N\geq 4$ be an integer. 
Denote by $\bP_N$ a regular $N$-gon, with vertices $\{1,2,\ldots, N\}$ modulo $N$ labeled in the counterclockwise direction.
An \emph{edge} of $\bP_N$ is a line segment joining two adjacent vertices.
Unless otherwise specified, a \emph{diagonal} is a line segment joining two non-adjacent vertices.
Denote by $[i,j]$ an edge or a diagonal with endpoints $i$ and $j$.
For any two diagonals, we say that they \emph{cross} each other or there is a \emph{crossing} between them if they are distinct and have a common interior point. 

When $N$ is even, we define a \emph{diameter} to be a diagonal joining two antipodal vertices. In this case, given a vertex $i$, we denote by $\overline i$ its antipodal vertex. We also use this notation for diagonals and edges, i.e. $\overline{[i,j]}=[\overline{i},\overline{j}]$.

For any vertices $i$ and $j$ of $\bP_N$, we define an \emph{arc} between $i$ and $j$ to be an undirected path on the boundary of the polygon connecting $i$ and $j$. 
The \emph{length} of an arc is the number of edges in the path.
Note that for any two vertices $i$, $j$ in $\bP_{N}$, there are two distinct arcs with endpoints $i$ and $j$.
If they have different lengths, the longer arc is called the \emph{major arc}, and the shorter arc is called the \emph{minor arc}.
We also define the \emph{length} of a diagonal $[i,j]$ to be the minimal length of an arc in $\bP_N$ between $i$ and $j$.
See Example \ref{Example_Diag_Arc_Universal} for an illustration of the notion of arcs and the length of a diagonal.

\begin{example}
	    \label{Example_Diag_Arc_Universal}
In Figure \ref{Figure_Example_Diam_Crossing}, edges of $\bP_8$ are denoted by dashed lines.
The diagonal $[1,5]$ is a diameter, while the diagonal $[3,8]$ is not.
The two diagonals $[1,5]$ and $[3,8]$ cross each other.
Figure \ref{Figure_Example_Diag_Arc} shows the diagonal $[1,6]$. 
Its minor arc and major arc are denoted by grey lines and black dashed lines respectively.
In particular, this diagonal is of length $3$.
    \begin{figure}[h]
        \centering
	\subcaptionbox{Crossing diagonals\label{Figure_Example_Diam_Crossing}
}{\includegraphics[height=4cm]{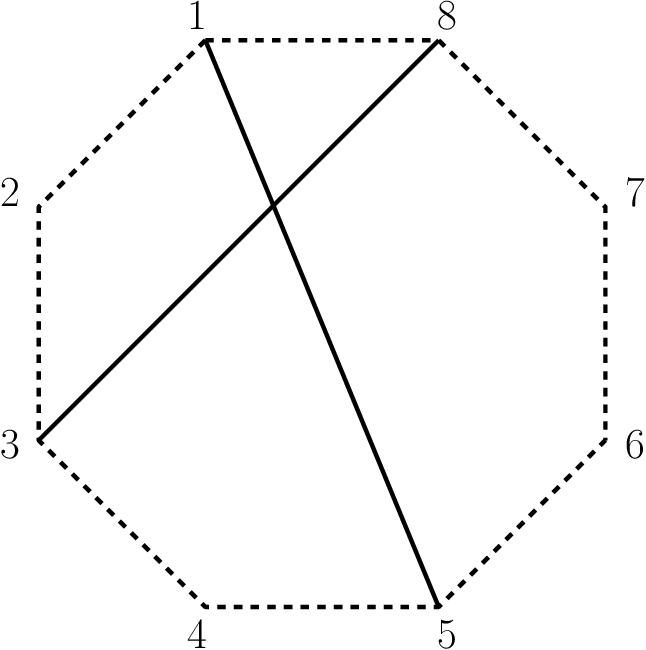}}\hspace{1cm}
	\subcaptionbox{A diagonal and its arcs\label{Figure_Example_Diag_Arc} }
	{\includegraphics[height=4cm]{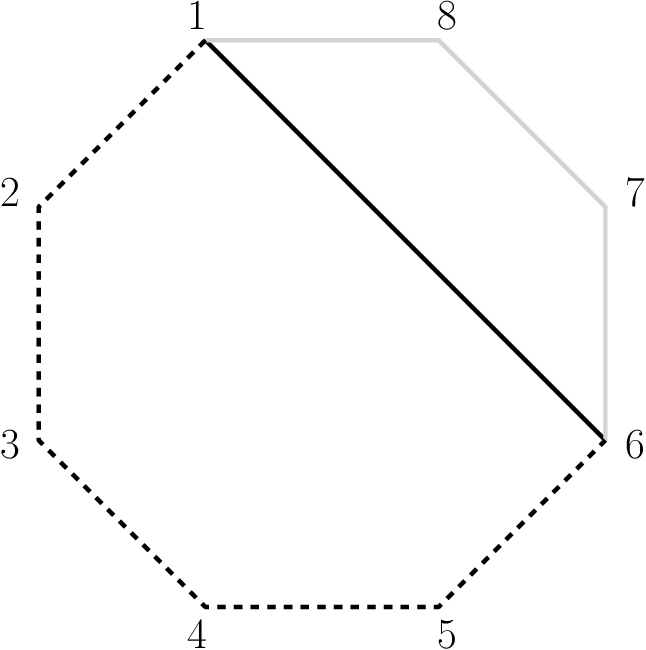}}
    \caption{The polygon $\bP_8$}
	    \end{figure}
\end{example}

\subsection{Type $A_n$}
We first describe the combinatorial model for type $A_n$. 
A \emph{triangulation} in $\bP_{N}$ is a maximal set of non-crossing diagonals of $\bP_{N}$.

\begin{proposition}[{\cite[Proposition 12.5]{Fomin_IIFiniteTypeClassification}}, {\cite[Proposition 3.14]{Fomin_YSystemsAndGeneralizedAssociahedra}}]\label{PropModelAn}

Let $n\geq 1$.
There exists a cluster algebra $\mathcal{A}$ of type $A_n$, a bijection between cluster variables of $\mathcal{A}$ and diagonals $\bP_{n+3}$, and a bijection between frozen variables of $\mathcal{A}$ and edges of $\bP_{n+3}$ satisfying the following:
\begin{enumerate}[label=(\arabic*)]
    \item 
    Clusters are in bijection with triangulations of $\bP_{n+3}$.
    
    \item 
    For a diagonal or an edge $l$, denote the corresponding cluster or frozen variable by $x_l$. 
    For diagonals or edges $l$ and $k$,
    \[
    (x_l||x_k)=\begin{cases}
        1 & \text{ if $l$ and $k$ cross,}\\
        0 & \text{ otherwise.}
    \end{cases}
    \]
    \item 
Given crossing diagonals $[a,c]$ and $[b,d]$, there is an
 associated exchange relation  
    \begin{equation}\label{Equation_Model_An}
        x_{[a,c]}x_{[b,d]}=x_{[a,b]}x_{[c,d]}+x_{[a,d]}x_{[b,c]},
    \end{equation} 
    and all exchange relations are of this form.
\end{enumerate}
\end{proposition}

\noindent In this setting, $\mathcal{A}$ is precisely the coordinate ring of the Grassmannian $\mathrm{Gr}(2,n+3)$ in its Plücker embedding. See \cite[Proposition 12.7]{Fomin_IIFiniteTypeClassification}.

\begin{example}\label{ex:A1}
	When $n=1$, the resulting cluster algebra $\mcA$ has cluster variables $x_{[1,3]},x_{[2,4]}$ and frozen variables $x_{[1,2]},x_{[2,3]},x_{[3,4]},x_{[1,4]}$. There is exactly one exchange relation
	\[
		x_{[1,3]}x_{[2,4]}=x_{[1,2]}x_{[3,4]}+x_{[2,3]}x_{[1,4]}.
	\]
	Taking the coordinates in $\RR^{\mcV\cup\mcW}$ in the order mentioned above, it is straightforward to calculate that the Gr\"obner cone $\mcA$ has lineality space generated by
\[
(0,1,1,1,0,0),\qquad (1,0,0,1,1,0),\qquad (0,1,0,0,1,1),\qquad (1,0,1,0,0,1)
\]
and rays generated by 
\[
(0,0,-1,0,0,0),\qquad (0,0,0,-1,0,0).
\]
\end{example}

\subsection{Types $B_n$ and $C_n$}
Next, we describe combinatorial models for types $B_n$ and $C_n$. These two types will share the same model, with small differences in the exchange relations. 
We will use the same choice of frozen variables for both type $B_n$ and type $C_n$. 

Consider the action of a $180^\circ$ rotation on the diagonals of $\bP_{2n+2}$. 
Each orbit of this action is either a diameter or a pair of centrally symmetric non-diameter diagonals $\{l,\overline{l}\}$.
We call these orbits \emph{diagonal pairs}, and diameters are considered as degenerate pairs, i.e., $\{l,\overline{l}\}$ with $l=\overline{l}$.

A \emph{centrally symmetric triangulation} in $\bP_{2n+2}$ is a triangulation that is fixed by the $180^\circ$ rotation action, or equivalently, a triangulation formed by non-crossing diagonal pairs. 
We set $r=1$ for type $B_n$ and $r=2$ for type $C_n$.

\begin{proposition}[{\cite[Proposition 12.9]{Fomin_IIFiniteTypeClassification}},{\cite[Proposition 3.15]{Fomin_YSystemsAndGeneralizedAssociahedra}}]\label{PropModelBCn}

Let $n\geq 2$ and $\Gamma=B_n$ or $\Gamma=C_n$.
There exists a cluster algebra $\mathcal{A}$ of type $\Gamma$, a bijection between cluster variables of $\mathcal{A}$ and the diagonal pairs of $\bP_{2n+2}$, and a bijection between frozen variables of $\mathcal{A}$ and centrally symmetric pairs of edges of $\bP_{2n+2}$ satisfying the following:

\begin{enumerate}[label=(\arabic*)]
    \item 
    Clusters are in bijection with centrally symmetric triangulations of $\bP_{2n+2}$.
    \item 
    For a diagonal pair or a pair of edges $\{l,\overline{l}\}$, denote the corresponding cluster or frozen variable by $x_l$ or $x_{\overline{l}}$.
    Let $L=\{l,\overline{l}\}$ and $K=\{k,\overline{k}\}$ be two diagonal pairs or pairs of edges.
    In type $B_n$, $(x_{l}||x_{k})$ equals the number of crosses between $l$ and the diagonal(s) in $K$. 
    In type $C_n$, the roles are reversed from the case of type $B_n$, i.e., $(x_{l}||x_{k})$ equals the number of crosses between $k$ and the diagonal(s) in $L$.
    \item 
The exchange relations are exactly the following:\\
    For $a,b,c,d,\overline a$ in counterclockwise order,
        \begin{equation}\label{Equation_Model_Bn_1}
        x_{[a,c]}x_{[b,d]}=x_{[a,b]}x_{[c,d]}+x_{[a,d]}x_{[b,c]};   
        \end{equation}
        
        For $a,b,c,\overline a$ in counterclockwise order,
        \begin{equation}\label{Equation_Model_Bn_2}
        x_{[a,c]}x_{[a,\overline{b}]}=x_{[a,b]}x_{[a,\overline{c}]}+x_{[a,\overline{a}]}^{2/r}x_{[b,c]}; 
        \end{equation}
       
        For $a\neq b$,
        \begin{equation}\label{Equation_Model_Bn_3}
        x_{[a,\overline{a}]}x_{[b,\overline{b}]}=x_{[a,b]}^{r}+x_{[a,\overline{b}]}^{r}.
        \end{equation}
\end{enumerate}

\end{proposition}

In type $B_n$, the cluster algebra $\mcA$ from Proposition \ref{PropModelBCn}
is the coordinate ring of the Grassmannian $\mathrm{Gr}(2,n+2)$ with a different choice of coordinates (\cite[Proposition 12.11]{Fomin_IIFiniteTypeClassification}).
In type $C_n$, $\mcA$ is the coordinate ring of the affine cone over the product $\mathbb{P}^n\times\mathbb{P}^n\subseteq\mathbb{P}^{n^2+2n}$ as the image of the Segre embedding (\cite[Example 6.3.4]{fomin2021introduction6}).

\subsection{Type $D_n$}
Finally, we describe a combinatorial model for type $D_n$. 
We again consider the orbits of the action of a $180^\circ$ rotation on the diagonals of $\bP_{2n}$, 
but each diameter can be of one of two different colors, denoted by $[a,\overline{a}]$ (called \emph{blue diameters} throughout this paper) and $\widehat{[a,\overline{a}]}$ (called \emph{red diameters} throughout this paper). We adopt the convention that two diameters may cross each other only if they are in different colors and have different endpoints. 
Hence, diameters of the same color do not cross (but diameters of any color may cross the non-diameter diagonals), and two common diameters of different colors do not cross as well.
We say that the non-diameters are \emph{plain diagonals} as opposed to the colored diameters.

A \emph{colored symmetric triangulation} in $\bP_{2n}$ is a maximal set of mutually non-crossing diagonal pairs and colored diameters. This is different from the usual definition of triangulations of a polygon because, in this kind of triangulation, we allow possible crossings between diameters of the same color.

\begin{proposition}[{\cite[Proposition 3.16]{Fomin_YSystemsAndGeneralizedAssociahedra}}, {\cite[Proposition 12.14]{Fomin_IIFiniteTypeClassification}}]\label{PropModelDn}
Let $n\geq 4$.
There exists a cluster algebra $\mathcal{A}$ of type $D_n$, a bijection between cluster variables of $\mathcal{A}$ and the set of all diagonal pairs and colored diameters of $\bP_{2n}$, and a bijection between frozen variables of $\mathcal{A}$ and the set of centrally symmetric pairs of edges of $\bP_{2n}$ satisfying the following:
\begin{enumerate}[label=(\arabic*)]
    \item 
   Clusters are in bijection with colored symmetric triangulations of $\bP_{2n}$.
    \item 
    For a diagonal pair or a colored diameter 
    $\{l,\overline{l}\}$, denote the corresponding cluster or frozen variable by $x_l$ or $x_{\overline{l}}$.
    Let $L=\{l,\overline{l}\}$ and $K=\{k,\overline{k}\}$ be diagonal pairs or colored diameters. 
    Then $(x_{l}||x_{k})$ is the number of crosses between $l$ and $k$ if both $l$ and $k$ are diameters. 
    Otherwise, $(x_{l}||x_{k})$ is half the number of crosses between the diagonal(s) in $L$ and in $K$.
    
    \item 
The exchange relations are exactly the following:

For $a,b,c,d,\overline{a}$ in counterclockwise order,
 \begin{equation}\label{Equation_Model_Dn_1}
    x_{[a,c]}x_{[b,d]}=x_{[a,b]}x_{[c,d]}+x_{[a,d]}x_{[b,c]};   
    \end{equation}

    For $a,b,c,\overline{a}$ in counterclockwise order,
    \begin{equation}\label{Equation_Model_Dn_2}
    x_{[a,c]}x_{[a,\overline{b}]}= x_{[a,b]}x_{[a,\overline{c}]}+ x_{[a,\overline{a}]}x_{\widehat{[a,\overline{a}]}}x_{[b,c]};    
    \end{equation}
    
    For $a,b,\overline{a}$ in counterclockwise order, 
    \begin{equation}\label{Equation_Model_Dn_3}
    x_{[a,\overline{a}]}x_{\widehat{[b,\overline{b}]}} = x_{[a,b]}+ x_{[a,\overline{b}]} ;   
    \end{equation}
    
    For $a,b,c,\overline{a}$ in counterclockwise order, 
    \begin{equation}\label{Equation_Model_Dn_4-1}
    x_{[a,\overline{a}]}x_{[b,\overline{c}]}= x_{[a,b]}x_{[c,\overline{c}]}+ x_{[a,\overline{c}]}x_{[b,\overline{b}]};    
    \end{equation}
    
    For $a,b,c,\overline{a}$ in counterclockwise order,    \begin{equation}\label{Equation_Model_Dn_4-2}
        x_{\widehat{[a,\overline{a}]}}x_{[b,\overline{c}]}= x_{[a,b]}x_{\widehat{[c,\overline{c}]}}+ x_{[a,\overline{c}]}x_{\widehat{[b,\overline{b}]}}.
    \end{equation}
\end{enumerate}
\end{proposition}

\noindent The cluster algebra $\mathcal{A}$ in Proposition \ref{PropModelDn} is the coordinate ring of the Schubert divisor of the affine cone of $\mathrm{Gr}(2,n+2)$(\cite[Example 6.3.5]{fomin2021introduction6}).

\subsection{Exchange Quadrilaterals}\label{sec:quadrilateral}
An important observation that we will use in our analysis is that any exchange relation for a cluster algebra of type $A_n$, $B_n$, $C_n$ or $D_n$ determines a special (pair of) quadrilaterals which we call \emph{exchange quadrilaterals}. We will use this extensively in subsequent sections. 
The definition of these quadrilaterals follows.
See Figure \ref{Figure_Exchange_Quads}  for a depiction of the exchange quadrilaterals for the various types of exchange relations.

\begin{figure}
	\centering
	\subcaptionbox{Type 0 exchange quadrilateral\label{Figure_Exchange_Quad_Type_0}
	}{\includegraphics[height=3.5cm]{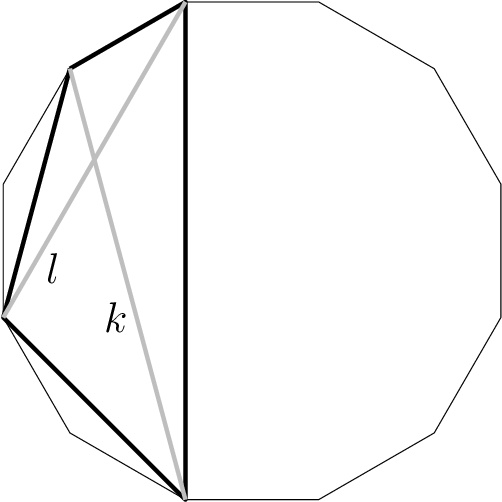}}\hspace{1cm}
	\subcaptionbox{Type 1 exchange quadrilaterals\label{Figure_Exchange_Quad_Type_1} }{\includegraphics[height=3.5cm]{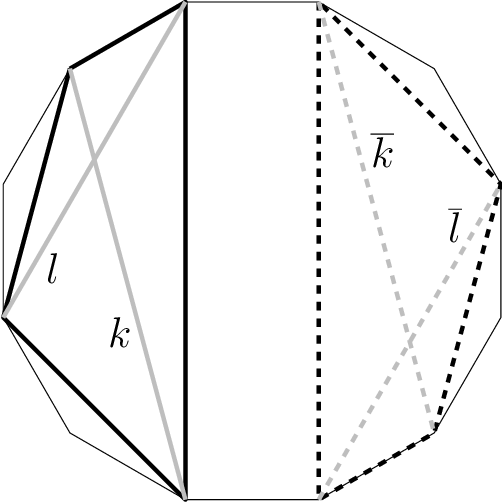}}\hspace{1cm}
	\subcaptionbox{Type 2 exchange quadrilaterals\label{Figure_Exchange_Quad_Type_2} }{\includegraphics[height=3.5cm]{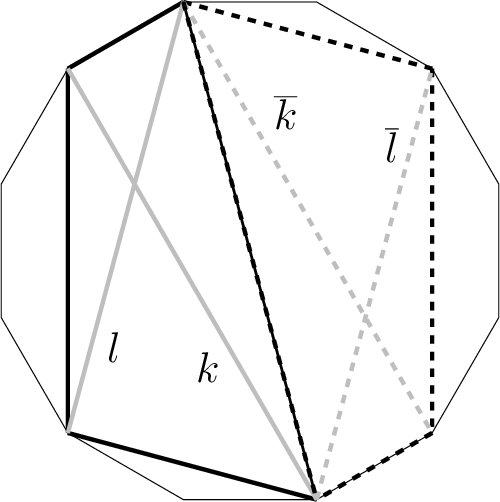}}\hspace{1cm}
	\subcaptionbox{Type 3 exchange quadrilaterals\label{Figure_Exchange_Quad_Type_3} }{\includegraphics[height=3.5cm]{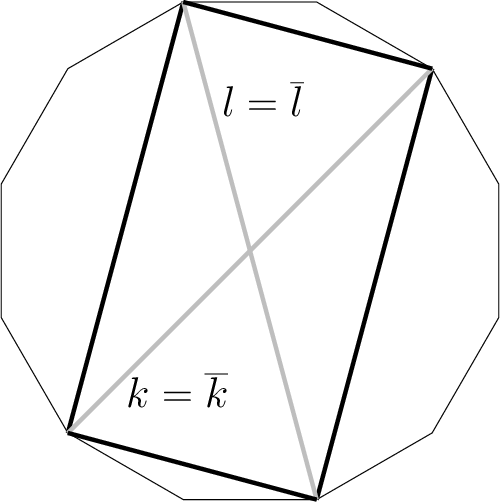}}\hspace{1cm}
	\subcaptionbox{Type 4 exchange quadrilaterals\label{Figure_Exchange_Quad_Type_4} }{\includegraphics[height=3.5cm]{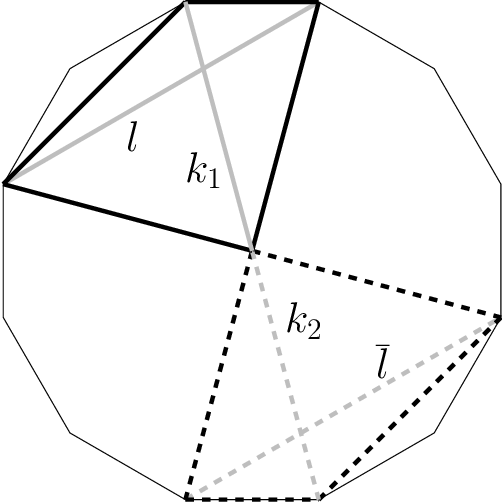}}
	\caption{Exchange quadrilaterals\label{Figure_Exchange_Quads}}
\end{figure}

\begin{figure}
	\centering
	\subcaptionbox{Type 1 exchange quadrilateral
	}{\includegraphics[height=4cm]{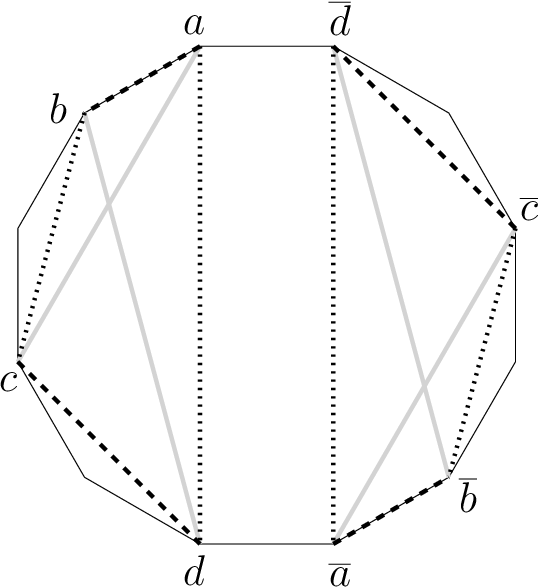}}\hspace{1cm}
	\subcaptionbox{Type 4 exchange quadrilaterals}
	{\includegraphics[height=4cm]{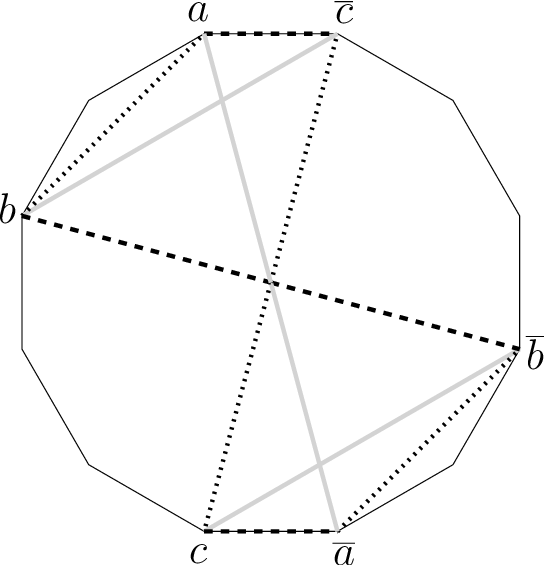}}
	\caption{Opposite edges of exchange quadrilaterals\label{Figure_Exchange_Quads_OppSide}}
\end{figure}

For exchange relations $P$ as in \eqref{Equation_Model_An}, the cluster variables being exchanged correspond in the combinatorial model to crossing diagonals $l$ and $k$; the associated exchange quadrilateral $\bQ$ is the convex hull of these diagonals. We call this a \emph{type 0} exchange quadrilateral.

For exchange relations $P$ as in \eqref{Equation_Model_Bn_1}, \eqref{Equation_Model_Bn_2}, \eqref{Equation_Model_Bn_3}, \eqref{Equation_Model_Dn_1}, \eqref{Equation_Model_Dn_2}, and \eqref{Equation_Model_Dn_3}, the cluster variables being exchanged correspond to pairs of crossing diagonals $\{l,\overline l\}$ and $\{k,\overline k\}$ with $l,k$ and $\overline l,\overline k$ respectively crossing. The associated pair of exchange quadrilaterals $\bQ,\overline \bQ$ are obtained as the convex hulls of $l,k$ and $\overline l,\overline k$.
If the pair $\bQ, \overline \bQ$ do not overlap, we say the pair is of \emph{type 1}; if they overlap at a diameter, we say it is of \emph{type 2}; if $\bQ = \overline \bQ$, we say it is of \emph{type 3}.

Finally, for exchange relations $P$ in type $D_n$ as in \eqref{Equation_Model_Dn_4-1}, \eqref{Equation_Model_Dn_4-2}, the cluster variables being exchanged correspond to a pair of diagonals $\{l,\overline l\}$ and a diameter $k$. Splitting the diameter $k$ at its midpoint into $k_1,k_2$ such that $l$ crosses $k_1$ and $\overline l$ crosses $k_2$, the pair of exchange quadrilaterals $\bQ,\overline \bQ$  associated to this exchange relation are the convex hulls of $l,k_1$ and $\overline l, k_2$.
We say that the pair $\bQ, \overline \bQ$ is of \emph{type 4}.

While the diagonals of the exchange quadrilaterals $\bQ,\overline \bQ$ for an exchange relation $P$ encode the cluster variables being exchanged in $P$,
the edges of the quadrilaterals encode the cluster and frozen variables appearing on the right hand side of the exchange relation. See Figure \ref{Figure_Exchange_Quads_OppSide}.
Indeed, consider an exchange relation $P$ with associated exchange quadrilateral(s) $\bQ$ (and $\overline \bQ$). For each pair of opposing edges of $\bQ$, the cluster and frozen variables corresponding to these edges are exactly the variables appearing in one of the two terms of right hand side of the exchange relation. Here, for relations of the form  \eqref{Equation_Model_Dn_4-1} or \eqref{Equation_Model_Dn_4-2}, the edges of $\bQ$ that are radii are treated as corresponding to the diameter.
Note that the multiplicity with which each variable appears in an exchange relation is not encoded by $\bQ$, and depends in particular on whether we are in type $B_n$ or $C_n$.

Using exchange quadrilaterals, it is easy to see which exchange relations are primitive: an exchange relation $P$ is primitive if and only if its exchange quadrilateral(s) $\bQ$ (and $\overline \bQ$) have at least one pair of opposing edges that are edges of $\bP_N$ (as opposed to diagonals). We will call any such quadrilateral a \emph{primitive} quadrilateral.
In this case, the cluster and frozen variables corresponding to the other pair of opposing edges of $\bQ$ are the variables appearing in a primitive term of $P$.
In particular, we see that a type 4 exchange in type $D_n$ is never primitive; this will allow us to ignore them in \S\ref{sec:frozen} and \S\ref{sec:nofrozen}.

\section{Term Orders from Compatibility Degree}\label{sec:TO}
\subsection{General Approach}

Our first main result is the following:
\begin{theorem}\label{thm:compdegree}
Let $\mcA$ be a cluster algebra of finite cluster type, $\mcV$ the set of its cluster variables, and $\mcW$ the set of its frozen variables.
Then for any $v\in \mcV$,
\[
\omega_v := \Big[(v||y)\Big]_{y\in \mathcal{V}\cup \mathcal{W}}\in \ZZ^{\mathcal{V}\cup \mathcal{W}}
\]
is an element of $\mcC_\mcA$, the Gröbner cone of $I_\mcA$ with respect to the initial ideal $I_\mcA^\cross$.
\end{theorem}

To prove Theorem \ref{thm:compdegree}, we will prove:
\begin{proposition}\label{prop:equality}
	Let $\mcA$ be a cluster algebra of finite cluster type. 
Let $P$ be an exchange relation 
\[
x\cdot x'=y_1+y_2
\]
with $x,x'$ exchangeable and $y_1$ and $y_2$ monomials in an extended cluster. Let $v$ be any cluster variable different from $x,x'$ and let $\omega_v$ be as in Theorem \ref{thm:compdegree}.
Then  
\begin{equation}\label{eqn:equality}
	\omega_v\cdot \deg (x\cdot x')= \max \{\omega_v\cdot \deg(y_1),\omega_v\cdot \deg(y_2)\}.
\end{equation}
\end{proposition}

To prove Proposition \ref{prop:equality}, we will first make two reductions. We first note that, without loss of generality, we may assume that $\mcA$ has no frozen variables. Indeed, since compatibility degree with any frozen variable is zero, the entries of $\omega_v$ indexed by frozen variables are zero, and so the frozen variables do not contribute to $\omega_v\cdot \deg(y_i)$.

We secondly observe that it suffices to prove Proposition \ref{prop:equality} in the case that $\mcA$ has an indecomposable exchange matrix $B$. Indeed, after simultaneously permuting rows and columns of $B$, we may assume that $B$ is block diagonal with indecomposable blocks. This block structure is preserved under mutation, and exchange relations only involve cluster variables corresponding to a single block. Hence we may consider each indecomposable block of $B$ separately. 

Since an exchange matrix $B$ being indecomposable implies that its Cartan counterpart is indecomposable, we have thus reduced to considering cluster algebras with no frozen coefficients and indecomposable Cartan matrices of finite type. These are classified exactly by Dynkin diagrams. Our proof now proceeds by an analysis of each type. For the classical types $A_n$, $B_n$, $C_n$, and $D_n$, we give a combinatorial argument in \S\ref{sec:classical}. For the exceptional types $E_6$, $E_7$, $E_8$, $F_4$, $G_2$, we verify the proposition computationally in \S\ref{sec:exceptional}.

Before diving further into the proof of Proposition \ref{prop:equality}, we show that the proposition almost immediately implies Theorem \ref{thm:compdegree}.

\begin{proof}[Proof of Theorem \ref{thm:compdegree}]
Consider any primitive exchange relation $P$ in $\mcA$ of the form $xx'=y_1+y_2$ with $y_1$ a primitive term. 
By Lemma \ref{Lemma_Of_The_Whole_Thesis}, we need to show that 
\[
\omega_v\cdot \deg(x\cdot x')\geq \omega_v\cdot \deg(y_1).
\]
If $v\neq x,x'$, then the desired inequality follows from Proposition \ref{prop:equality}.
If instead $v=x$, then $v$ is compatible with all the variables in $y_1$, hence $\omega_v\cdot \deg(y_1)=0$; the same is true if $v=x'$. Since compatibility degree is always non-negative, we always have $\omega_v\cdot \deg(x\cdot x')\geq 0$ and the desired inequality again follows. 
\end{proof}

	\subsection{Classical Types}\label{sec:classical}
In this section, we prove Proposition \ref{prop:equality} for the classical types $A_n$, $B_n$, $C_n$, or $D_n$ with no frozen variables. For type $A_1$ with no frozen variables there is nothing to prove, since there is only a single exchange relation of the form $x\cdot x'=1+1$. We will henceforth assume that $n\geq 2$, and $n\geq 4$ in type $D_n$.

For a convex quadrilateral $\bQ$ in the plane, denote its crossing diagonals by $\bQ(0)$, and its two pairs of opposite edges by $\bQ(1)$ and $\bQ(2)$ respectively.
Denote the number of crosses between a line segment $l$ and a set of line segments $S$ by $l\cdot S$. 
Two coincident segments do not count as a cross.
We first prove the following combinatorial fact.

\begin{lemma}\label{Lemma_CombFact_QuadCrossing}
	Let $\bQ$ be a convex quadrilateral. 
	Then for any line $l$ that does not contain a segment of $\bQ(0)$,
	\begin{equation*}
		l\cdot \bQ(0)=\max_{i=1,2}\{l\cdot \bQ(i)\}.
	\end{equation*}
	Moreover, if for some $i>0$ we have $l\cdot \bQ(i)>0$, then
	\[
		l\cdot \bQ(0)=l\cdot \bQ(i).
	\]
\end{lemma}

\begin{proof}
	If $l$ does not cross any line of $\bQ$, both claims are trivially true.
	Otherwise, without loss of generality, there are three possible positions of $l$ relative to $\bQ$ as shown in Figure \ref{Figure_Lemma_CombFact_QuadCrossing}.
	By inspection, the claims follow in each case.
	\begin{figure}
		\centering
		{\includegraphics[height=4cm]{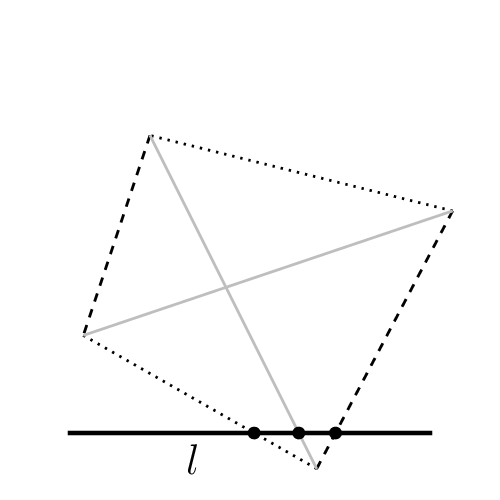}}\hspace{.2cm}
		{\includegraphics[height=4cm]{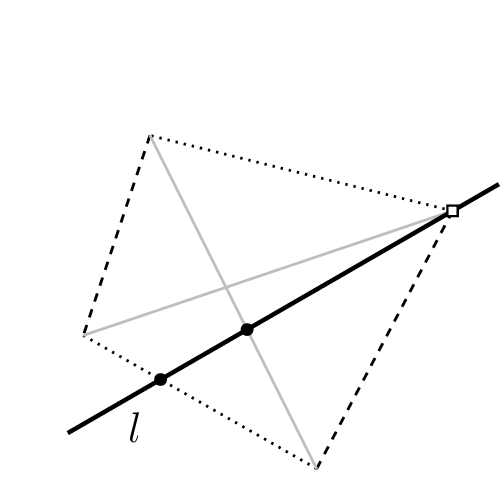}}\hspace{.2cm}
		{\includegraphics[height=4cm]{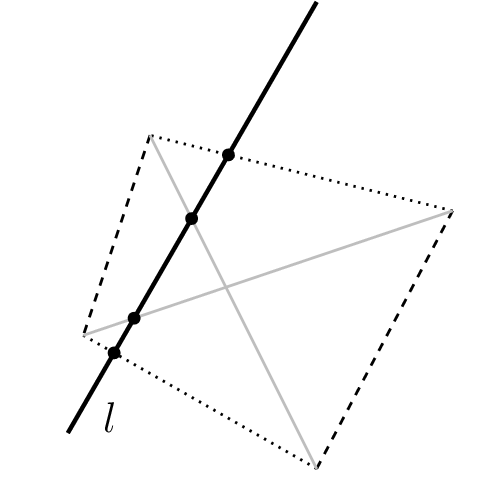}}
		\caption{Possible positions of $l$ relative to $\bQ$\label{Figure_Lemma_CombFact_QuadCrossing}}
	\end{figure}
\end{proof}

For an exchange relation $xx'=y_1+y_2$ with exchange quadrilateral(s) $\bQ$ or $\{\bQ,\overline{\bQ}\}$, we denote $y_0=x x'$ so that $\bQ(i)$ or $\{\bQ(i),\overline{\bQ(i)}\}$ is the set of diagonals that correspond to $y_i$.
We use the following lemma which relates the combinatorial fact above to the computation of the dot product $\omega_v\cdot \deg(y_i)$.
\begin{lemma}\label{lemma:weighteddegree}
	Let $\mcA$ be a cluster algebra with no frozen variables of type $A_n$, $B_n$,  or $C_n$ with $n\geq 2$, or type $D_n$ with $n\geq 4$.
	Let $P$ be an exchange relation
	\[
	x x' = y_1 + y_2
	\] 
	with corresponding exchange quadrilateral(s) $T=\{\bQ\}$ or $T=\{\bQ,\overline{\bQ}\}$, and denote $y_0 = xx'$. 
	Let $v\in\mcV$ be a cluster variable different from $x$ and $x'$ with $v$ corresponding to diagonal(s) $L=\{l\}$ or $L=\{l,\overline{l}\}$.
	Then, if $P$ is of exchange type 0, 1, 2, or 3, there exists a natural number $\epsilon$ such that 
	\begin{align}
		\omega_v \cdot \deg(y_i) & = \frac{1}{\epsilon}\sum_{\bR\in T,\ k\in L}k\cdot \bR(i), \qquad i=0,1,2;\label{eqn:newx}
	\end{align}
	If $\mcA$ is of type $D_n$ and $P$ is of exchange type 4, either \eqref{eqn:newx} holds, or $l$ is a diameter and
	\begin{equation}\label{eqn:weighted_deg_equals_crossing4}
		\omega_v \cdot \deg(y_i) = l\cdot \bQ(i) + D(l,i),\qquad i=0,1,2
	\end{equation}
	where $D(l,i)$ denotes the number of diameters in $y_i$ that cross $l$ and have a different color from $l$.
\end{lemma}

\begin{proof}
	We split up the cases by considering the type of $\mcA$.
	Denote the compatibility degree in a cluster algebra of type $\Gamma$ by $(\cdot || \cdot)_{\mathrm{\Gamma}}$.
	
	\subsubsection*{Type $A_n$}
	Suppose $\mathcal{A}$ is of type $A_n$. 
	By Proposition \ref{PropModelAn}, the compatibility degree is given by
	\[
	(x_a||x_b)_{A_n}=(x_b||x_a)_{A_n}=
	\begin{cases}
		1 & \text{if the diagonals } a \text{ and  } b \text{ cross,}\\
		0 & \text{otherwise.}
	\end{cases}
	\]
	Moreover, in every exchange relation, the cluster variables appearing have degree 1. 
	Since in type $A_n$ all exchange quadrilaterals are of type 0, the degree of a variable in $P$ equals the number of lines corresponding to that variable in $\bQ$. 
	Furthermore, $v$ corresponds to a single diagonal $l$, so $L=\{l\}$, and $T=\{\bQ\}$.
	This implies that \eqref{eqn:newx} holds for type $A_n$ with $\epsilon=1$.
	
	\subsubsection*{Types $B_n$ and $C_n$}
	Suppose $\mathcal{A}$ is of type $B_n$ or type $C_n$.
	Set $r=1$ for type $B_n$ and $r=2$ for type $C_n$.
	Recall from Proposition \ref{PropModelBCn} that, if $A = \{a,\overline{a}\}$ and $B = \{b,\overline{b}\}$ are sets of diagonals, the compatibility degrees are given by
	\begin{align*}
		(x_a||x_b)_{B_n} &=\text{number of crosses between $a$ and $B$},\\
		(x_a||x_b)_{C_n} &=\text{number of crosses between $A$ and $b$}.
	\end{align*}
	There are three types of exchange quadrilaterals in $\mathcal{A}$, namely, type 1 (corresponding to Equation \eqref{Equation_Model_Bn_1}), type 2 (corresponding to Equation \eqref{Equation_Model_Bn_2}), and type 3 (corresponding to Equation \eqref{Equation_Model_Bn_3}).
	We will show \eqref{eqn:newx} holds for some $\epsilon$ in each type.
	
	\begin{figure}
		\centering
		\subcaptionbox{Type $B_n$\label{fig:Pf_Bn_Type_2_Exchange}}{\includegraphics[height=4cm]{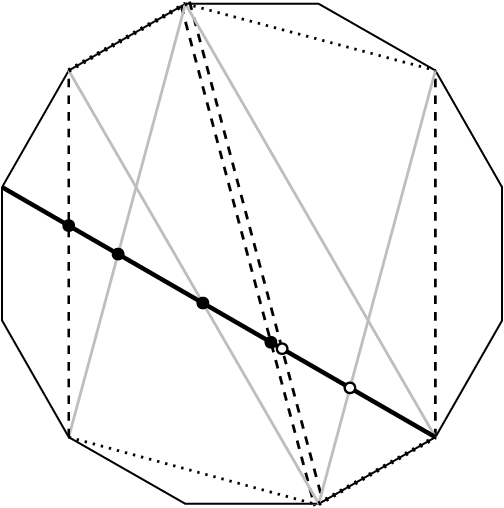}}\hspace{1cm}
\subcaptionbox{Type $C_n$\label{fig:Pf_Cn_Type_2_Exchange}}{\includegraphics[height=4cm]{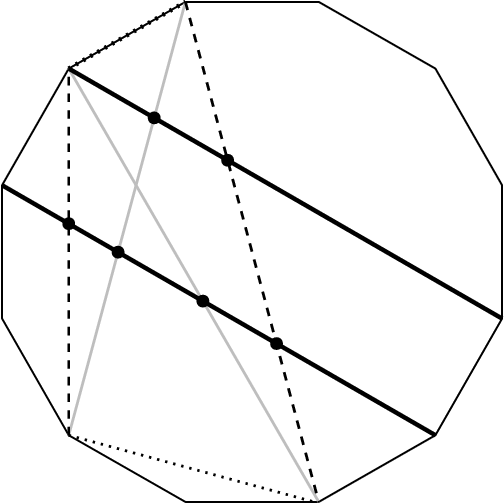}}
\caption{Single and double lines in type 2 exchanges}
	\end{figure}

	\subsubsection*{Type $B_n$}
	For type $B_n$, by the definition of compatibility degree, $\omega_v \cdot \deg(y_i)$ equals the number of crosses between $l$ and the diagonals in $\bQ(i), \overline{\bQ(i)}$, where each cross is weighted by the exponent of the corresponding variable in $P$.
	In type 1 and type 3 exchanges, all exponents are $1$ (or $0$), so 
	\[
		\omega_v \cdot \deg(y_i)=\sum_{\bR\in T} l\cdot \bR(i).
	\]
	This means we can take $\epsilon=\#L$ to obtain \eqref{eqn:newx}.
	
	In type 2 exchanges, the overlapping of $\mathbf{Q}$ and $\overline{\mathbf{Q}}$ at a diameter is accounted for by the exponent $2$ of the corresponding variable in $P$, so we treat this diameter as a double line. See Figure \ref{fig:Pf_Bn_Type_2_Exchange}.	
	Again, \eqref{eqn:newx} holds with $\epsilon=\#L$.

	\subsubsection*{Type $C_n$}
	For type $C_n$, by the definition of compatibility degree, 
	$\omega_v \cdot \deg(y_i)$ equals the number of crosses between $L$ and one set of diagonal representatives of $y_i$ in $\bQ(i)$ and $\overline{\bQ(i)}$, and each cross is weighted by the degree of the corresponding variable in $P$.
	Therefore, in type 1 exchanges, we immediately see that 
	\eqref{eqn:newx} holds with $\epsilon=2$.
	
	For type 2 exchanges, the variable corresponding to the diameter has exponent $1$ in $P$.
	As opposed to what we see in type $B_n$, there is a single line instead of a double line where $\mathbf{Q}$ and $\overline{\mathbf{Q}}$ intersect.
	Therefore, we only need to count crosses between $L$ and one of $\{\mathbf{Q}, \overline{\mathbf{Q}}\}$, and hence \eqref{eqn:newx} still holds with $\epsilon=2$.
	 
	For type 3 exchanges, $\bQ = \overline{\bQ}$, and so the $i=0$ case of \eqref{eqn:newx} holds with $\epsilon=1$. 
	Both variables $y_1$ and $y_2$ have exponent $2$ in $P$, and they both correspond to non-diameters.
	We claim that computing the compatibility degree with $y_1$ or $y_2$ is equivalent to counting the number of crosses of $L$ with both diagonals of the pair corresponding to $y_1$ or $y_2$. 
		Indeed, if $l$ crosses both diagonals in $\bQ(i)$, the claim is immediate.
	Otherwise, if $l$ only crosses one of the diagonals in $\bQ(i)$ but not the other, $\overline{l}$ must cross the other diagonal in $\bQ(i)$.
	This is because the set $L$ of diagonals is invariant under the $180^\circ$ rotation, but the diagonals in $\bQ(i)$ are interchanged under this rotation. See Figure \ref{fig:Pf_Cn_Type_3_Exchange}.
	Therefore, the weighted count of crossings of $L$ with one set of diagonal representative of $\bQ(i)$ is equal to the number of crossings between $L$ and $\bQ(i)$.
It follows that \eqref{eqn:newx} still holds with $\epsilon=1$.
	
	\begin{figure}
		\centering
		\subcaptionbox{Weighted count of crossings.}
		{\includegraphics[height=4cm]{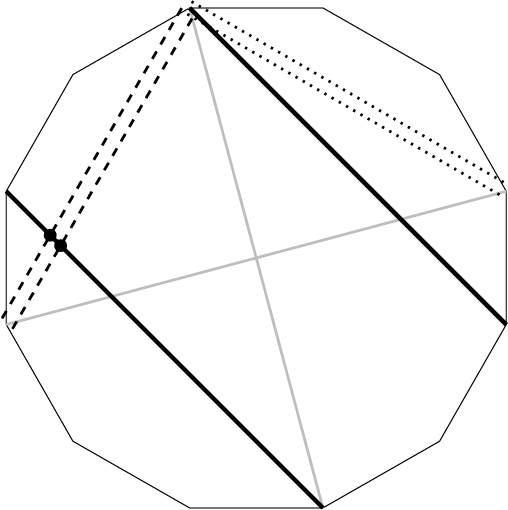}}\hspace{1cm}
		\subcaptionbox{Crosses between $L$ and $\bQ(i)$. } 
		{\includegraphics[height=4cm]{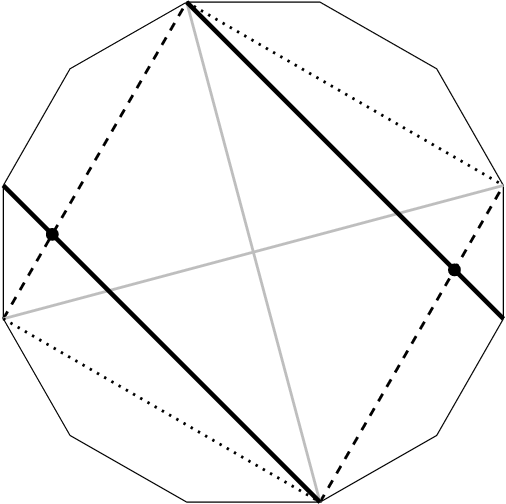}}
		\caption{Equivalence of cross counting in type 3 for $C_n$\label{fig:Pf_Cn_Type_3_Exchange}}
	\end{figure}	
	
	\subsubsection*{Type $D_n$}
	Suppose $\mathcal{A}$ is of type $D_n$. 
	Recall from Proposition \ref{PropModelDn} that if $A = \{a,\overline{a}\}$ and $B = \{b,\overline{b}\}$ are sets of diagonals, the compatibility degrees are given by
	\[
	(x_a||x_b)_{D_n} =
	\begin{cases}
		\text{number of crosses between $a$ and $b$} & \text{if } a \text{ and } b \text{ are diameters,}\\
		\dfrac{1}{2}(\text{number of crosses between $A$ and $B$}) & \text{otherwise,}\\
	\end{cases}
	\]
	with the convention that diameters of the same color do not cross.

	Suppose first that $\#L=2$ and we have an exchange relation of type 1, 2, or 3. Then the analysis from type $B_{n-1}$ applies and we have \eqref{eqn:newx} holding with $\epsilon=2$. Notice that in type 3, the power of $2$ in type $B_{n-1}$ for the cluster variable associated to the common diameter of $\bQ$ and $\overline{\bQ}$ is replaced by two cluster variables corresponding to the two different colors of that diameter.
	
	\begin{figure}
		\centering
				\subcaptionbox{$|L|=2$\label{fig:Pf_Dn_Type_4_Exchange1}}
				{\includegraphics[height=4cm]{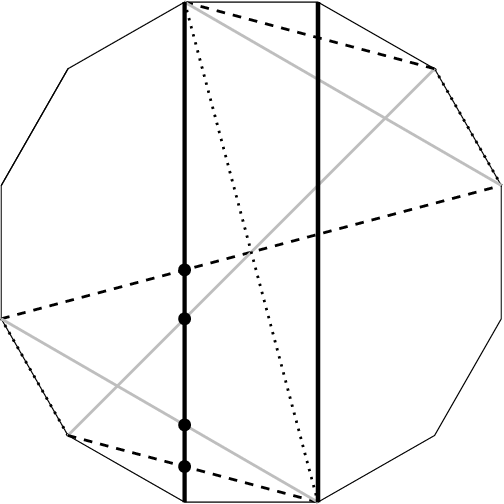}}\hspace{1cm}
			\subcaptionbox{$|L|=1$\label{fig:Pf_Dn_Type_4_Exchange2}}{\includegraphics[height=4cm]{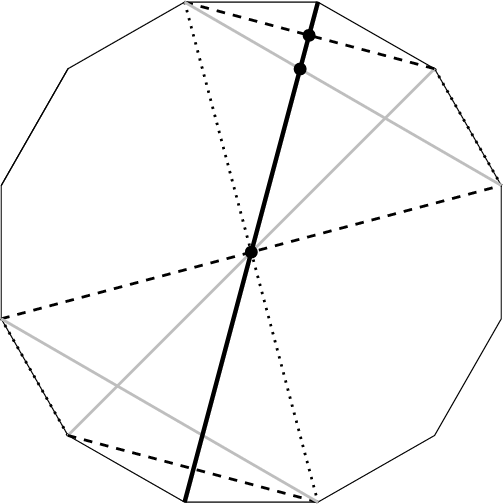}}
			\caption{Crossing in type 4 for $D_n$}
	\end{figure}
	If $\#L=2$ and $P$ is a type 4 exchange, note that if $l$ crosses a line $k$ in $\bQ$, it does not cross $k$ again in $\overline{\bQ}$, since  $l$ and $k$ cross at most once.
	Hence,
	\begin{align*}
		\omega_v\cdot \deg(y_i) & = (v||y_i)_{D_n}\\
		& = \dfrac{1}{2}\big(\text{number of crosses between $L$ and $\{\bQ(i),\overline{\bQ(i)}\}$}\big)
	\end{align*}
	and so \eqref{eqn:newx} holds with $\epsilon=2$.
	See Figure \ref{fig:Pf_Dn_Type_4_Exchange1}.

We now suppose $\#L=1$, that is, $l$ is a diameter, which we may assume without loss of generality to be blue. 
	Note that since $l$ is blue, it does not cross blue diameters.
	If $P$ is of type 1 or type 2, the analysis from type $C_{n-1}$ applies and we have \eqref{eqn:newx} holding with $\epsilon=2$.

	Suppose $P$ is a type 3 exchange.
Since the diagonals of $\bQ(0)$ are diameters of different colors, 
\[
		\omega_v \cdot \deg(y_0)  = \dfrac{1}{2}(l\cdot \bQ(0)).
	\]
But for $i>0$ we also have
\[
		\omega_v \cdot \deg(y_i)  = \dfrac{1}{2}(l\cdot \bQ(i))
	\]
	so \eqref{eqn:newx} holds with $\epsilon=2$.

Finally, suppose $P$ is a type 4 exchange. Then $l$ crosses the diameters appearing in $\bQ$ and $\overline{\bQ}$ at the center of $\bP_N$, but since the center is a vertex of both $\bQ$ and $\overline{\bQ}$, this cross is not counted in the expressions $l\cdot\bQ(i)$ and $l\cdot \overline{\bQ(i)}$, but the crosses between $l$ and the non-diameters are counted.
	Moreover, since $l$ is a diameter, if $l$ crosses a non-diameter in $\bQ$, it also crosses the counterpart of that diagonal in $\overline{\bQ}$.
	Since the compatibility degree is given by half the number of crosses when a diameter crosses a diagonal pair, counting two crosses from each of $\bQ$ and $\overline{\bQ}$ is the same as counting only the cross in $\bQ$.
	See Figure \ref{fig:Pf_Dn_Type_4_Exchange2}.
	Hence, 
	\[
	\omega_v\cdot \deg(y_i) = l\cdot \bQ(i) + D(l,i),
	\]
	which is precisely Equation \eqref{eqn:weighted_deg_equals_crossing4}.

\end{proof}

We now prove Proposition \ref{prop:equality} for the classical types:
\begin{proof}[Proof of Proposition \ref{prop:equality} in types $A_n,B_n,C_n,D_n$]
	We use notation as in the statement of Lemma \ref{lemma:weighteddegree}.
	Suppose first that $\#L=1$. 
	Then for exchanges of types 0, 1, 2, and 3 the claim of the proposition then follows from Lemma \ref{Lemma_CombFact_QuadCrossing} and Lemma \ref{lemma:weighteddegree}. Indeed, if $\#T=1$ this is immediate. If instead $\#T=2$, we use that  $l\cdot \bQ(i)=l\cdot \overline \bQ(i)$ since $l=\overline l$.
	For exchanges of type 4, we note that $l$ must be a diameter, and then $D(l,0)=D(l,1)=D(l,2)$; the claim of the proposition again follows.

	Suppose instead that $\#L=2$. If $\#T=1$ (i.e. we have a type 3 exchange) then $l\cdot \bQ(i)=\overline l\cdot \bQ(i)$ since $\bQ=\overline \bQ$, so the claim of the proposition follows from Lemma \ref{Lemma_CombFact_QuadCrossing} and Lemma \ref{lemma:weighteddegree}.

	It remains to consider the case that $\#L=\#T=2$.
	If $l$ intersects $\bQ$ but not $\overline \bQ$, then for all $i$, $l\cdot \bQ(i)=\overline l\cdot \overline \bQ(i)$ and $\overline l\cdot \bQ(i)=l \cdot \overline \bQ(i)=0$, and again the claim of the proposition follows from Lemmas
\ref{Lemma_CombFact_QuadCrossing} and \ref{lemma:weighteddegree}.
Suppose instead that $l$ intersects both $\bQ$ and $\overline \bQ$. If we have an exchange relation of type 1 or 2, then $l$ must intersect the longest edges of $\bQ$ and $\overline \bQ$, and the claim follows similar to above. See Figure \ref{fig:Pf_Type1or2_intersecting_both_quads}.
If we have an exchange relation of type 4, then for either $i=1$ or $i=2$, $l$ must intersect both $\bQ(i)$ and $\overline\bQ(i)$; see Figure \ref{fig:Pf_Type4_intersecting_both_quads}.
The claim now follows with a similar argument to above.

\begin{figure}
	\centering
	\includegraphics[height=4cm]{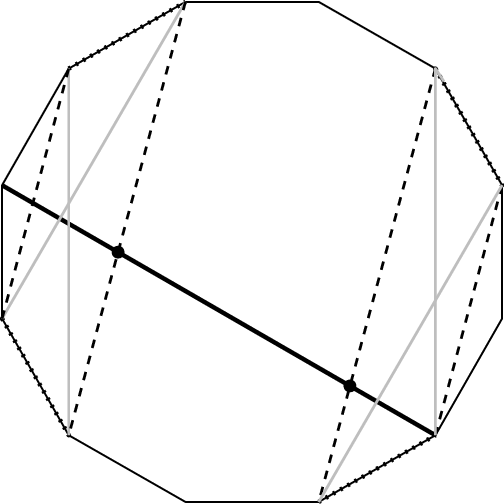}
	\caption{$l$ intersecting both $\bQ$ and $\overline{\bQ}$ \label{fig:Pf_Type1or2_intersecting_both_quads}}
\end{figure}

\begin{figure}
	\centering
	\includegraphics[height=4cm]{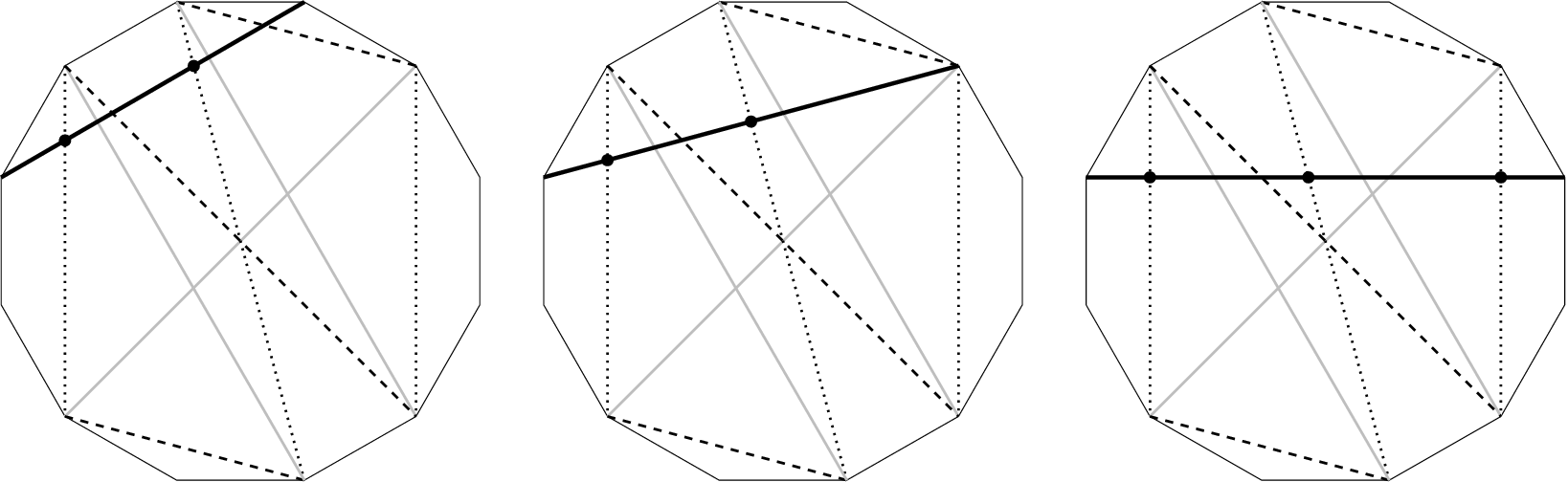}
	\caption{$l$ intersecting both $\bQ$ and $\overline{\bQ}$ in a type 4 exchange \label{fig:Pf_Type4_intersecting_both_quads}}
\end{figure}

\end{proof}

\subsection{Exceptional Types}\label{sec:exceptional}
In this section, we discuss the computation verifying Proposition \ref{prop:equality} for the exceptional types $E_6$, $E_7$, $E_8$, $F_4$, and $G_2$. For the computation, we use SageMath \cite{sagemath}, including built-in functions for cluster algebras, Weyl groups, undirected graphs, and linear algebra. Accompanying computer code is available at \cite{ancillary}.

In order to verify Proposition \ref{prop:equality}, we need to understand compatibility degrees and exchange relations in the exceptional types. For compatibility degrees, we use Theorem \ref{Theorem_c-compdeg_def} and the discussion immediately thereafter to compute compatibility degrees in terms of $\bg$-vectors. Note that in creating the initial exchange matrix $B(c)$, we use the Coxeter element $c=s_1\cdots s_n$, using SageMath's ordering of the simple reflections in the Weyl group of the given type. Using the iterative process of Theorem \ref{Theorem_c-compdeg_def}, we obtain a matrix $C$ with rows and columns indexed by cluster variables (represented as elements of $\Pi(c)$) whose entries are the respective compatibility degrees of the pair.
The vectors $\omega_v$ from Theorem \ref{thm:compdegree} are exactly the columns of $C$.

From the matrix $C$, we are able to see which pairs of cluster variables are exchangeable. For any exchangeable pair $x,x'$, we then use Theorem \ref{Theorem_GenPrimExchRel} to obtain the corresponding exchange relation $P$.
More specifically, for the pair of weights $\lambda=\psi(x),\mu=\psi(x')\in \Pi(c)$, we obtain the weights of all cluster monomials appearing in $P$ by computing the elements $\{\tau_c^{-m}(\tau_c^{m}(\lambda)+\tau_c^{m}(\mu))\}_{m\in\ZZ}$ until two distinct elements $\lambda+\mu$ and $\lambda\uplus_c\mu$ are returned. 

To determine the actual cluster monomials corresponding to $\lambda+\mu$ and $\lambda\uplus_c\mu$, 
we use the compatibility degree matrix $C$ to find a set $S_{\lambda, \mu}\subseteq \Pi(c)$ such that $S_{\lambda, \mu}$ is a maximal set that is compatible with both $\lambda$ and $\mu$, and that the elements in $S_{\lambda, \mu}$ are pairwise compatible.
 This is done by finding a maximum clique of the graph whose vertices are elements of $\Pi(c)$ compatible with $\lambda$ and $\mu$, and with edges given by pairs of compatible elements.

 By definition of $S_{\lambda,\mu}$, the exchange between $x$ and $x'$ occurs between the two clusters corresponding under $\psi$ to $S_{\lambda,\mu}\cup\{\lambda\}$ and $S_{\lambda,\mu}\cup\{\mu\}$. 
Since $\psi$ is a bijection from cluster monomials to the weight lattice, we may uniquely represent $\lambda+\mu$ and $\lambda\uplus_c\mu$ as non-negative integral linear combinations of the elements of $S_{\lambda,\mu}\cup\{\lambda\}$.
We obtain the cluster variables appearing in the right hand side of $P$ as $\psi^{-1}(S_{\lambda,\mu})$, and the exponents are exactly the integral coefficients in the above linear combinations.

Verifying Proposition \ref{prop:equality} is now simply a matter of testing \eqref{eqn:equality} for every cluster variable $v$ and for every $P$ coming from a pair of exchangeable cluster variables $x,x'$.
\subsection{Derivation Degrees}
Let $\mcA$ be a cluster algebra of finite cluster type with cluster variables $\mcV$ and frozen variables $\mcW$. Let $I_\mcA^\cross$ be the monomial ideal from Definition \ref{defn:Icross}, and $\mcC_\prim$ the cone from Lemma \ref{Lemma_Of_The_Whole_Thesis}.
In \cite[Conjecture 6.2.9]{ilten2021deformation}, Ilten, Nájera Chávez, and Treffinger conjecture that the multidegrees of first order embedded deformations of $\spec (\KK[\bz]/I_\mcA^\cross)$ that become trivial when forgetting the embedding do not interact with the semigroup generated by multidegrees of first order deformations induced by the universal cluster algebra of $\mcA$. We provide a precise statement of this conjecture and show that it follows from our Proposition \ref{prop:equality}.

For a cluster variable $v\in \mathcal{V}$ and a monomial $\bz^\alpha\in \KK[\bz]$, we define 
\[
	\partial(v,\alpha):=\bz^{\alpha}\cdot\frac{\partial}{\partial z_v}\in \Der_\KK(\KK[\bz,\bz])
\]
where $\dfrac{\partial}{\partial z_v}$ is the formal derivative by $z_v$. We also define
\[
\deg(\partial(v,\alpha)):= \deg(\bz^{\alpha})-\deg(z_v)\in \ZZ^{\mathcal{V}\cup \mathcal{W}}.
\]
We say that $\partial(v,\alpha)$ is $I_\mcA^\cross$-non-trivial if its image in $\Hom_{\KK[\bz]}(I_\mcA^\cross,\KK[\bz]/I_\mcA^\cross)$ is non-zero, or
equivalently, if there exists a cluster variable $w$ such that 
\[
z_w\cdot z_v \in I_\mcA^\cross \hspace{0.5cm}\text{ and }\hspace{0.5cm} z_w\cdot \bz^{\alpha} \not\in I_\mcA^\cross.
\]
The following corollary gives some necessary conditions of the $I_\mcA^\cross$-non-trivial property.

\begin{corollary}[{\cite[Conjecture 6.2.9]{ilten2021deformation}}]\label{Corollary_ModCompDeg}
  Let $\mcA$ be a cluster algebra of finite cluster type.
	If $\partial(v,\alpha)$ is $I_\mcA^\cross$-non-trivial, then 
    \[
    -\deg(\partial(v,\alpha))\not\in \mcC_\prim.
    \]
\end{corollary}

\begin{proof}
	Let $v\in \mcV$ be a cluster variable, and let $\omega_v\in \ZZ^{\mcV\cup \mcW}$ be the vector from Theorem \ref{thm:compdegree}.
We define $\omega_v'=\omega_v-e_v$, that is, we change the $v$-entry of $\omega_v$ from $0$ to $-1$.

We claim that $\omega_v'$ satisfies
\[
\omega_v'\cdot \deg(x\cdot x')\geq \omega_v'\cdot \deg(y_1)
\]
    for any exchange relation $P$ of the form $x\cdot x'=y_1+y_2$.
   
    Indeed, if $v\neq x,x'$, by Proposition \ref{prop:equality} we have
    \begin{align*}
\omega_v'\cdot \deg(x\cdot x')=\omega_v\cdot \deg(x\cdot x')
\geq \omega_v\cdot \deg(y_1)\\=(\omega_v'+e_v)\cdot \deg (y_1)\geq \omega_v'\cdot y_1.
\end{align*}

If instead $v=x$
\[
\omega_v'\cdot \deg(x\cdot x')=-1+(v||x')=-1+1=0
\]
since $x=v$ and $x'$ are exchangeable and hence have compatibility degree $1$.
On the other hand, $v=x$ is compatible with all variables in $y_1$, so 
\[
\omega_v'\cdot \deg(y_1)=0
\]
and the desired inequality follows. A similar argument works for $v=x'$.
By our claim above, we may now conclude that $\omega_v'$ is in the cone dual to $\mcC_\prim$.

    We  now show that if $\partial(v,\alpha)$ is $J$-non-trivial, $\omega_v' \cdot \deg(\partial(v,\alpha))>0$.
    This would imply that $-\deg(\partial(v,\alpha))\not \in \mcC_\prim$, proving the corollary. To that end, note that 
    since $\partial(v,\alpha)$ is $I_\mcA^\cross$-non-trivial, $z_v$ is not a factor of $\bz^\alpha$. 
    In both $\omega_v'$ and $\deg(\partial(v,\alpha))$, the only negative entry is the entry indexed by $v$, and both have $v$-entry $-1$. Therefore 
    \[
    \omega_v' \cdot \deg(\partial(v,\alpha))\geq (-1)(-1) =1 >0
    \]
as desired.
\end{proof}

 \begin{remark}
	 As originally formulated,  \cite[Conjecture 6.2.9]{ilten2021deformation} is slightly weaker than our Corollary \ref{Corollary_ModCompDeg}. Indeed, the original claim is that $-\deg(\partial(v,\alpha))$ does not belong to the \emph{semigroup} whose generators are the elements in Lemma \ref{Lemma_Of_The_Whole_Thesis} that we took to generate $\mcC_\prim$.
\end{remark}

\subsection{Gr\"obner Cone Generators}
In Theorem \ref{thm:compdegree}, we have identified a collection of vectors $\omega_v$ that belong to the Gr\"obner cone $\mcC_\mcA$. We can use these
to produce an element in the \emph{interior} of $\mcC_\mcA$, giving rise to a circular term order. 
\begin{corollary}\label{Corollary_ExplicitCircTermOrder}
    Let $\mcA$ be a cluster algebra of finite cluster type. Then
    \[
    \omega=\sum_{v\in \mcV}[(v||y)]_{y\in \mathcal{V}\cup\mathcal{W}}
    \]
    belongs to the interior of the Gr\"obner cone $\mcC_\mcA$.
    \end{corollary}

\begin{proof}[Proof of Corollary \ref{Corollary_ExplicitCircTermOrder}]
    By Lemma \ref{Lemma_Of_The_Whole_Thesis}
we must show that 
    \[
    \omega \cdot  d >0
    \]
    for any degree $d=\deg (x\cdot x')-\deg (y_1)$ induced by a primitive exchange relation $x\cdot x'=y_1+y_2$ with primitive term $y_1$.
   
    For every $v\in \mcV$, 
    \[
    \omega_v \cdot d \geq 0
    \]
    by Theorem \ref{thm:compdegree}.
    Moreover, if $v=x$ or $v=x'$, $\omega_v \cdot d=1$ since $v$ is compatible with one of $x,x'$ and exchangeable with the other, and compatible with all variables in $y_1$.

    We conclude that
    \[
	    \omega \cdot d =\bigg(\sum_{v\in \mcV} \omega_v \bigg)\cdot d \geq 2
    \]
    for any degree $d$ induced by a primitive exchange relation.
\end{proof}

In some situations, we can even use the vectors $\omega_v$ to readily describe the rays of the Gr\"obner cone $\mcC_\mcA$. 
Consider any permutation  $T$ on the cluster variables $\mcV$ such that for every cluster variable $x\in \mcV$, the cluster variables $x,T(x)$ are the exchangeable variables in a primitive exchange relation
\begin{equation}\label{eqn:T}
    x \cdot T(x) = y_x + y_x'
\end{equation}
with primitive term $y_x$
 and every primitive exchange relation arises this way.
 By Theorem \ref{Theorem_GenPrimExchRel}, such a permutation exists: we can take $T=\psi^{-1}\circ \tau_c\circ \psi$.
Note that although $T$ is not in general unique, its orbits in $\mcV$ are independent of the choice of $T$.

\begin{remark}\label{rem:tau}
	 Interpreting cluster variables in types $A_n$, $B_n$, $C_n$, and $D_n$ as (pairs of) diagonals in $\bP_N$, we may take $T$ to be the permutation $[i,j]\mapsto [i+1,j+1]$ (which switches colors of diameters in type $D_n$). Indeed, this follows from the description in \S\ref{sec:quadrilateral} of exchange quadrilaterals corresponding to primitive exchange relations.
 \end{remark}

 \begin{remark}\label{rem:orbits}
For a cluster algebra $\mcA$ of finite cluster type with indecomposable exchange matrix $B$, a permutation $T$ as above has no even-sized orbits in $\mcV$ if and only if $\mcA$ is of type $A_n$, $B_n$, or $C_n$ where $n$ is even, or type $F_4$.
Indeed, for the classical types, this follows from Remark \ref{rem:tau}. For the exceptional types, this may be checked using SageMath \cite{sagemath}, see \cite{ancillary}.
 \end{remark}

\begin{theorem}\label{Therorem_Result2_Main}
    Let $\mcA$ be a cluster algebra with no frozen variables of type $A_n$, $B_n$, $C_n$, or $F_4$, where $n$ is even.
    Let $\mcV$ be the set of its cluster variables, and let $T$ be a permutation of $\mcV$ as above.
Then the Gr\"obner cone $\mcC_\mcA$ is a pointed cone whose rays are generated by 
\begin{equation*}
	\widehat\omega_v=\sum_{i=1}^k (-1)^{i-1} \cdot \omega_{T^i(v)},
\end{equation*}
where $k$ is the size of the $T$-orbit of $v$, and $v$ ranges over all $v\in\mcV$.
\end{theorem}
\begin{proof}Using the notation of \eqref{eqn:T}, 
set \[
d_x=\deg (x\cdot T(x))-\deg(y_x). 
\]
	We will show that for any cluster variables $v,x\in\mcV$, 
    \begin{equation*}
    \widehat\omega_v \cdot d_x= \begin{cases}
        2 & \text{ if }x=v',\\
        0 & \text{ otherwise.}
    \end{cases}
    \end{equation*}
    It will then follow from Lemma \ref{Lemma_Of_The_Whole_Thesis} that $\mcC_\mcA$ is pointed and has rays that are generated by the $\widehat\omega_v$.

    For any $v,x\in \mcV$ we have
    \[\omega_v \cdot d_x=\begin{cases}
0& v\neq x,T(x)\\
1& v=x\ \textrm{or}\ v=T(x).
	    \end{cases}
\]
Indeed, the first case follows from Proposition \ref{prop:equality}. The second case holds since $v$ is compatible with one of $x,T(x)$ and exchangeable with the other, and compatible with all variables in $y_x$.

Fix now $v\in\mcV$.	By Remark \ref{rem:orbits}, the orbit of $v$ under $T$ has odd size $k$.
For $x=v$, we have 
\[
	\widehat\omega_v\cdot d_x=\omega_{T(v)}\cdot d_x+(-1)^{k-1} \omega_{v}\cdot d_x=1+1=2.
\]
For $x=T^j(v)$ with $1\leq j < k$ we have
\[
	\widehat\omega_v\cdot d_x=(-1)^{j-1}\omega_{T^j(v)}\cdot d_x+(-1)^{j} \omega_{T^{j+1}(v)}\cdot d_x=(-1)^{j-1}+(-1)^{j}=0.
\]
Finally, if $x$ is not in the $T$-orbit of $v$, then clearly $\widehat\omega_v\cdot d_x=0$. 
\end{proof}

\section{Classical Types with Frozen Variables}\label{sec:frozen}
\subsection{Lineality Spaces}\label{sec:linspace}
In this section, we will explicitly describe generators for the Gr\"obner cones $\mcC_\mcA$ when $\mcA$ is one of the cluster algebras with frozen variables described by the combinatorial models in \S\ref{sec:comb}.
We begin in this subsection by describing the lineality spaces of $\mcC_\mcA$.

Let $\mcA$ be the cluster algebra of type $A_n$, $B_n$ or $C_n$ ($n\geq 2$), or $D_n$ ($n \geq 4$) with frozen variables described in \S\ref{sec:comb}. Let $N$ be $n+3$, $2n+2$, $2n+2$, or $2n$, respectively. Recall that the cluster variables $\mcV$ of $\mcA$ are in bijection with diagonals or diagonal pairs of $\bP_N$, where in type $D_n$ diameters are colored blue or red. The frozen variables $\mcW$ of $\mcA$ are in bijection with edges or edge pairs of $\bP_N$.

We now introduce some notation for vectors in $\RR^{\mcV\cup \mcW}$. For a diagonal or edge in type $A_n$ or (colored) diameter $l$ in types $B_n,C_n,D_n$, let $e_l\in \RR^{\mcV\cup \mcW}$ denote the standard basis vector corresponding to $l$. Similarly, for a diagonal or edge pair $L=\{l,\overline l\}$ in types $B_n,C_n,D_n$ let $e_L$ denote the standard basis vector corresponding to $L$. We will occasionally abuse notation and for $L=\{l,\overline l\}$ write $e_L=e_l=e_{\overline l}$.

In type $A_n$, for any vertex $i$ of $\bP_N$ we define
\[
	E_i=\sum_{l\ |\ i\in l} e_l
\]
where the sum is taken over all edges or diagonals $l$ that are incident to the vertex $i$. Similarly, in types $B_n,D_n$ we define
\[
	 E_i=\sum_{l:\ i\in l} e_l+\sum_{L=\{l,\overline l\}:\ i\in l\cup\overline l} e_L
\]
and in type $C_n$
\[
	E_i=e_{[i,\overline i]}+\sum_{l:\ i\in l} e_l+\sum_{L=\{l,\overline l\}:\ i\in l\cup\overline l} e_L
\]
where the sums are taken over all (colored) diameters $l$, or edge or diagonal pairs $L$, that are incident to $i$.

\begin{proposition}\label{Proposition_ABCn_LinSp}
    Let $\mcA$ be the cluster algebra of type $A_n$ ($n\geq 1$), or $B_n$ or $C_n$ ($n\geq 2$) with the frozen variables described in \S\ref{sec:comb}.
    For $\Gamma=A_n$, $B_n$, or $C_n$, define $K=K(\Gamma)$ by 
    \begin{align*}
        K(\Gamma)&=\bigg\{E_i ~\bigg|~i=1,\ldots, n+3\bigg\},\qquad \Gamma=A_n\\
        K(\Gamma)&=\bigg\{E_i ~\bigg|~i=1,\ldots, n+1\bigg\},\qquad \Gamma=B_n,C_n\\
    \end{align*}
    with notation as above.
    Then $K$ is a basis of the lineality space of the Gröbner cone $\mcC_\mcA$ of $\mcA$.
\end{proposition}
\noindent We will prove this proposition in \S\ref{sec:proofs1}.

In Figure \ref{Figure_Example_Lsp_Elt} we depict a general element in the set $K$ in types $A_n$ (left), $B_n$ (middle), and $C_n$ (right). Solid lines denote diagonals or edges corresponding to coordinates with entry $1$, and the dash-dot-dotted line denotes entry $2$.

\begin{figure}
	\centering
	{\includegraphics[height=4cm]{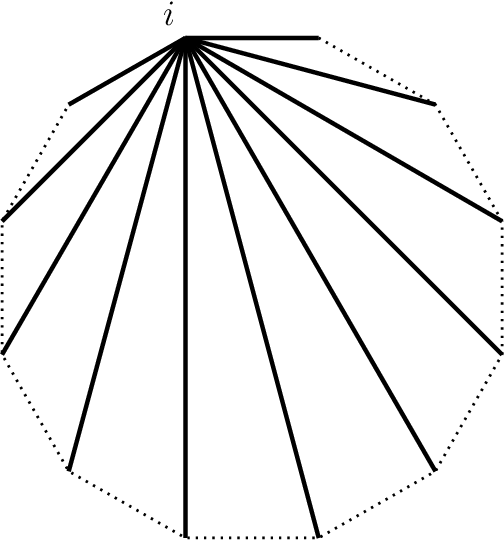}}\hspace{.5cm}
	{\includegraphics[height=4cm]{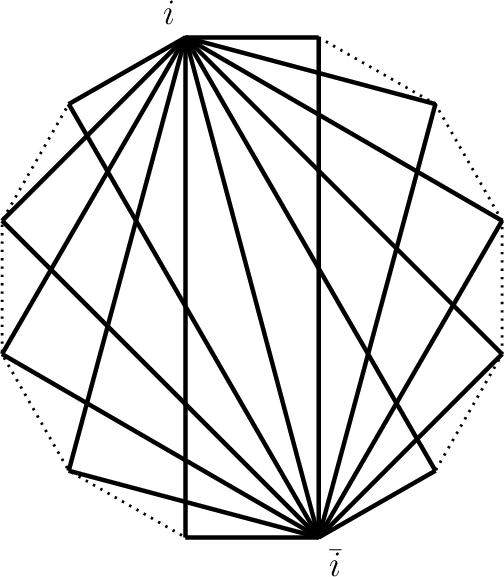}}\hspace{.5cm}
	{\includegraphics[height=4cm]{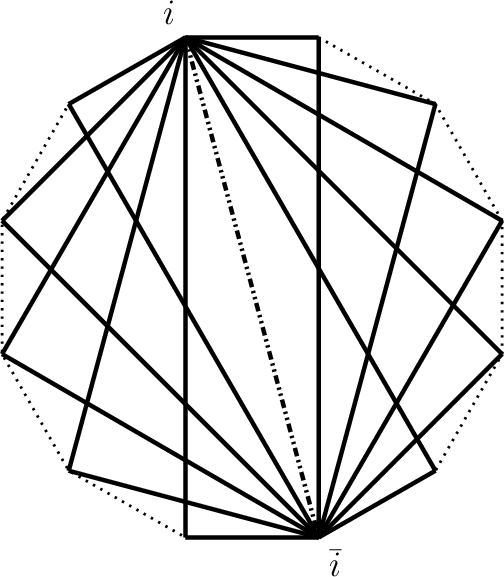}}
	\caption{Generators of the lineality space in types $A_n$, $B_n$, and $C_n$}
	\label{Figure_Example_Lsp_Elt}
\end{figure}

To describe the lineality space in type $D_n$ we define 
\[
	u^\diam=
	\sum_{\substack{l\ \textrm{a blue}\\\textrm{diameter}}} e_l
	-\sum_{\substack{l\ \textrm{a red}\\\textrm{diameter}}} e_l.
\]

\begin{proposition}\label{Proposition_Result3_Dn_Linsp}
	For $n\geq 4$, let $\mcA$ be the cluster algebra of type $D_n$ with the frozen variables described in \S\ref{sec:comb}.
    Then the set 
    \[
	    K(D_n)=\bigg\{{E}_i ~\bigg|~i=1,\ldots, n\bigg\}\cup \{u^{\mathrm{diam}}\}\subseteq \ZZ^{\mathcal{V}\cup\mathcal{W}},
    \]
    is a basis of the lineality space of the Gröbner cone $\mcC_\mcA$ of $\mcA$.
\end{proposition}
\noindent We will prove this proposition in \S\ref{sec:proofs1}.

In Figure \ref{Figure_Lsp_Dn} we depict various elements of $K(D_n)$.
The picture on the left shows ${E}_i$ and the one on the right shows $u^\diam$.
Red and blue diameters are depicted by red and blue lines; other diagonals or edges are black. Coordinates with $+1$ entries correspond to solid lines, and coordinates with $-1$ entries correspond to dashed lines.

\begin{figure}
    \centering    
    \includegraphics[height=4cm]{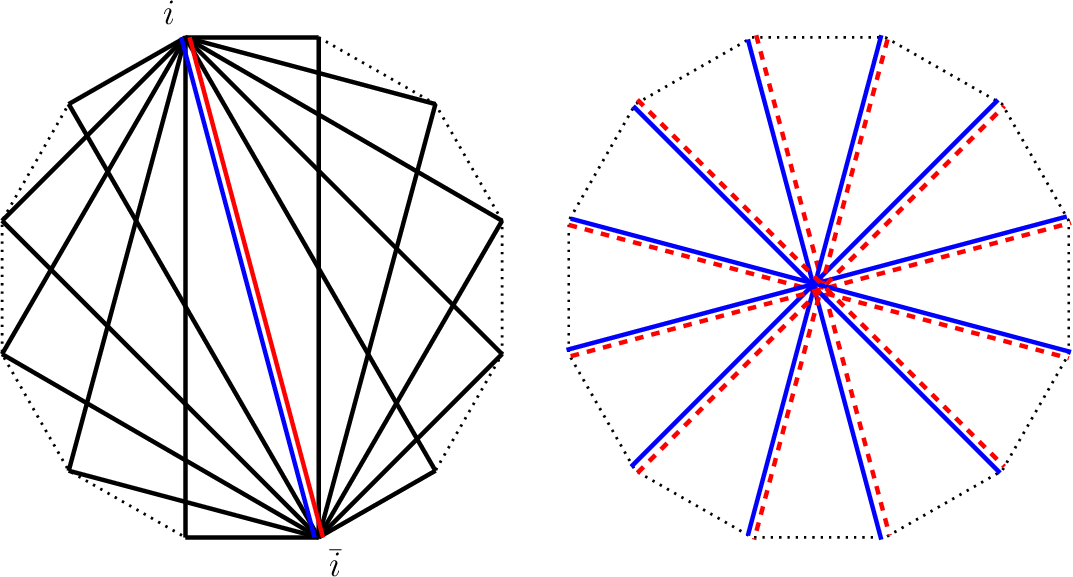}   
    \captionsetup{justification=centering}
    \caption{Generators of the lineality space in type $D_n$.}
    \label{Figure_Lsp_Dn}
\end{figure}

\subsection{Ray Generators}\label{sec:rays}
We now describe the generators for the rays of $\mcC_\mcA$.
We say an edge or diagonal of $\bP_N$ is of \emph{non-maximal-length} if the length of $l$ is less than $\lfloor N/2\rfloor $.
For such $l$, let $F_l$ be the set of diagonals and edges of $\bP_N$ with both endpoints contained in the minor arc of $l$.

In type $A_n$ we then set
\[
	v(l)=-\sum_{k\in F_l} e_{k}.
\]
In types $B_n,C_n,D_n$ we set
\[
	v(l)=-\sum_{k\in F_l} e_{\{k,\overline {k}\}}.
\]

\begin{theorem}\label{Theorem_Statement_Result3_ABCn}
	Let $\mcA$ be the cluster algebra of type $A_n$ ($n\geq 1$), or $B_n$ or $C_n$ $(n\geq 2)$ with the frozen variables described in \S\ref{sec:comb}.
    Then the elements of the set 
    \[
	    \mcG=\{v(l):l \text{ is a diagonal or edge of }\bP_{N} \text{ of non-maximal length}\}
    \]
    generate the rays of the Gröbner cone $\mcC_\mcA$ modulo lineality space.
\end{theorem}
\noindent We will prove this Theorem in \S\ref{sec:proofs2}. In types $B_n$ and $C_n$, clearly $v(l)$ and $v(\overline l)$ generate the same ray modulo lineality space. We will also see that this is the case in type $A_n$ when $n$ is odd and $l$ has length $(n+1)/2$.

\begin{figure}
    \centering    
    \includegraphics[height=4cm]{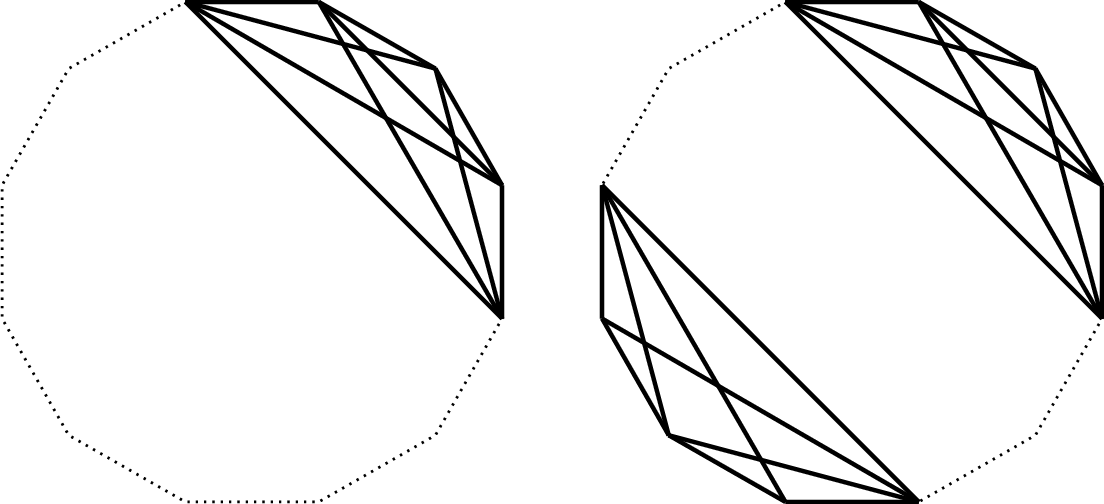}   
    \captionsetup{justification=centering}
    \caption{Generators of the Gröbner cone in types $A_n$, $B_n$, and $C_n$}
    \label{Figure_Example_GCone_Gen}
\end{figure}
In Figure \ref{Figure_Example_GCone_Gen} we depict some vectors $v_l$ in type $A_n$ (left) and in types $B_n$ and $C_n$ (right). 
Black solid lines indicate diagonals and edges corresponding to coordinates with entry $-1$.

For type $D_n$ we need more notation.
For $i=1,\ldots, n$ set 
\[
	w(i)=e_{[i,\overline i]}+v([i+1,\overline i-1])\qquad \widehat{w}(i)=e_{\widehat{[i,\overline i]}}+v([i+1,\overline i-1]).
\]
For $1\leq j\leq n$ and $0\leq k \leq n-1$ we then set
\[
	w(j,k)=\sum_{i=0}^{k/2} w(j+2i)-\sum_{i=1}^{k/2} \widehat{w}(j+2i-1)
\]
if $k$ is even and
\[
	w(j,k)=	v([j+k,\overline j+k+1])+\sum_{i=0}^{(k-1)/2 } w(j+2i)-\sum_{i=1}^{(k+1)/2} \widehat{w}(j+2i-1)
\]
if $k$ is odd.

\begin{theorem}\label{Theorem_Statement_Result3_Dn}
	For $n\geq 4$, let $\mathcal{A}$ be the cluster algebra of type $D_n$ with the frozen variables described in \S\ref{sec:comb}.
    Then the elements in the union of the sets
	\[
	\{v(l): l \text{ is a diagonal or edge of $\bP_N$ of length at most $n-2$}\}
\]
and 
	\[
		\{w(j,k)\ :\ 1\leq j\leq n, 0\leq k\leq n-1\}
\]
    generate the rays of the Gröbner cone $\mcC_\mcA$ modulo lineality space.
\end{theorem}
\noindent We will prove this in \S\ref{sec:proofs2}. Notice that $v(l)$ generates the same ray as $v(\overline l)$.

In Figure \ref{Figure_Constr_Dn} we depict some generators of the rays of the Gr\"obner cone $\mcC_\mcA$ in type $D_n$.
The top left picture is an element $v(l)$ constructed from a diagonal of length at most than $n-2$.
The top right, bottom left and bottom right pictures show the elements $w(1,0)$, $w(1,1)$ and $w(1,2)$ respectively.
Red and blue diameters are depicted by red and blue lines; other diagonals or edges are black. Coordinates with $+1$ entries correspond to solid lines, and coordinates with $-1$ entries correspond to dashed lines.

\begin{figure}
    \centering    
    \includegraphics[height=4cm]{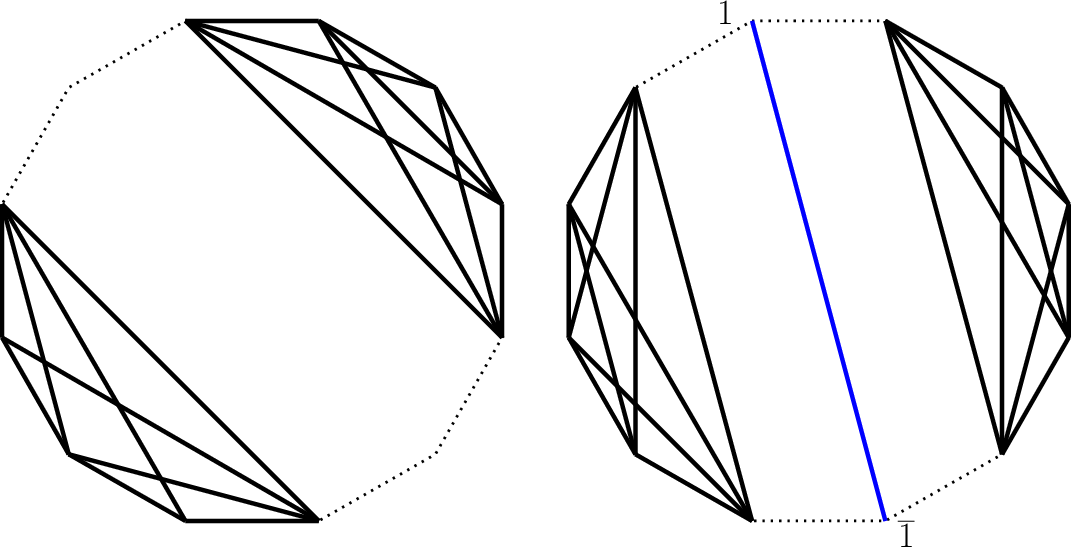}\vspace{1cm}
    \includegraphics[height=4cm]{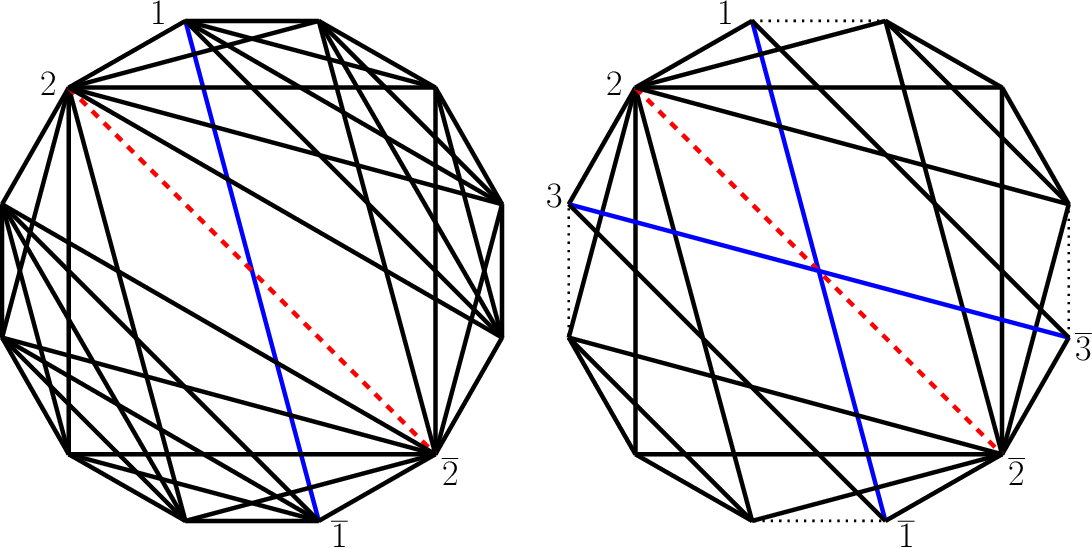}  
    \captionsetup{justification=centering}
    \caption{Generators of the Gröbner cone in type $D_n$}
    \label{Figure_Constr_Dn}
\end{figure}

\subsection{Proofs of Propositions  \ref{Proposition_ABCn_LinSp} and \ref{Proposition_Result3_Dn_Linsp}}
\label{sec:proofs1}
The case of $A_1$ follows from Example \ref{ex:A1}, so for the remainder of this subsection and the next we will always assume that $n\geq 2$. Most importantly, any primitive exchange relation $P$ has a unique primitive term, and we have the degree $d_P$ induced by the primitive exchange relation $P$.

We first determine the dimension of the lineality space of the Gr\"obner cone:
\begin{lemma}\label{Lemma_Nullity_q}
    Let $\mcA$ be the cluster algebra of type $A_n$, $B_n$, $C_n$,  with frozen variables as in \S\ref{sec:comb}.
    The dimension of the lineality space of the Gr\"obner cone $\mcC_\mcA$ is equal to the number of frozen variables. For $\mcA$ the cluster algebra of type $D_n$ with frozen variables as in \S\ref{sec:comb}, the dimension of the lineality space is one more than the number of frozen variables.
\end{lemma}
\begin{proof}
	By \cite[Proposition 5.3.1]{ilten2021deformation}, the dimension of the lineality space of $\mcC_\mcA$  is equal to the dimension of the stabilizer of $\mcA$ in the multigraded Hilbert scheme of \cite[\S 5.1 and 5.2]{ilten2021deformation}.
	By \cite[Proposition 5.1.1]{ilten2021deformation}, the dimension of the stabilizer is equal to the dimension of the cokernel of any extended exchange matrix. But for the cluster algebras $\mcA$  of type $A_n$, $B_n$, $C_n$ we are considering, any extended exchange matrix has full rank, see e.g.~\cite[Example 5.2.11]{ilten2021deformation}.
    Hence, the cokernel of any extended exchange matrix has dimension $m-n$, which is exactly the number of frozen variables.

    In the $D_n$ case we consider it is straightforward to verify that any extended exchange matrix has rank $n-1$, so the stabilizer dimension is one larger than the number of frozen variables.
\end{proof}

Next, we need some more combinatorics.
Let $\bQ$ be an exchange quadrilateral in $\bP_N$. As in \S\ref{sec:classical}, denote the diagonals of $\bQ$ by $\bQ(0)$. 
Recall that $\bQ$ is primitive if it has (at least) one pair of opposing edges that are both edges in $\bP_N$. Denote the other pair of opposing edges by $\bQ(1)$. 

    \begin{figure}
        \centering    
        \includegraphics[height=4cm]{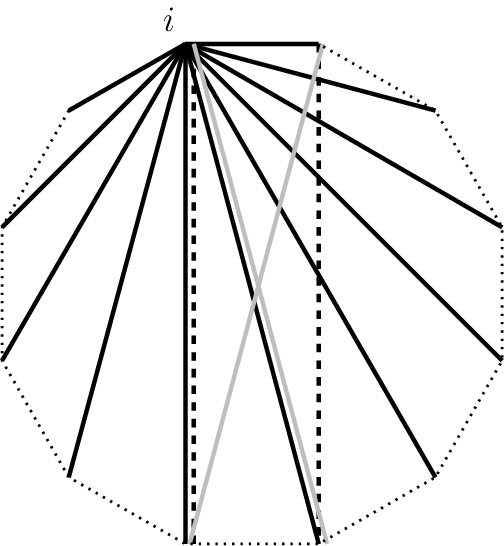}   
        \captionsetup{justification=centering}
	\caption{Primitive quadrilaterals in Lemma \ref{Lemma_FixingVertex_is_in_LinealitySpace}}
        \label{Figure_Fix_Vertex_LinSp}
    \end{figure}
\begin{lemma}\label{Lemma_FixingVertex_is_in_LinealitySpace}
    Let $\bQ$ be a primitive quadrilateral in a regular polygon $\bP_N$.
    Let $F_i$ be the set of all diagonals incident to a fixed vertex $i$ of $\bP_N$.
    Then
    \[
    \#(\bQ(0)\cap F_i)=\#(\bQ(1)\cap F_i).
    \]
\end{lemma}

\begin{proof}
    If $i$ is not a vertex of $\bQ$, the statement is trivially true.
    If $i$ is a vertex of $\bQ$, there is only one possible arrangement of $\bQ$ with respect to $F_i$ up to symmetry, as shown in Figure \ref{Figure_Fix_Vertex_LinSp}. 
    Lines in $F_i$ are denoted by black solid lines, elements of $\bQ(0)$ by solid gray lines, and elements of $\bQ(1)$ by dashed black lines. 
    The claim follows by inspection.
\end{proof}

    \begin{lemma}\label{Claim_512}
	    Let $\mcA$ be the cluster algebra of type $A_n$, $B_n$, or $C_n$ with $n\geq 2$ or $D_n$ with $n\geq 4$ and with frozen variables as in \S\ref{sec:comb}. 	    Let $\bQ$ (and $\overline\bQ$) be the quadrilateral(s) of a primitive exchange relation $P$. Then one of the following holds:
        \begin{enumerate}[label=(\arabic*)]
            \item \label{Claim515_Expression1}
            $E_i\cdot d_P = \#(\bQ(0)\cap F_i)-\#(\bQ(1)\cap F_i)$; or
            \item \label{Claim515_Expression2}
            $E_i\cdot d_P = \#(\bQ(0)\cap F_i)-\#(\bQ(1)\cap F_i)+ \#(\overline\bQ(0)\cap F_i)-\#(\overline\bQ(1)\cap F_i)$; or
            \item \label{Claim515_Expression3}
		    $E_i\cdot d_P = \#(\bQ(0)\cap F_i)-\#(\bQ(1)\cap F_i)+ \#(\bQ(0)\cap F_{\overline{i}})-\#(\bQ(1)\cap F_{\overline{i}})$.
	    \end{enumerate}
    \end{lemma}
    \begin{proof}
    We proceed to prove the claim splitting into types.
    In the figures we reference below, elements of $\bQ(0)$ are indicated by gray solid lines and elements of $\bQ(1)$ by dashed black lines.

    \subsubsection*{Type $A_n$}
    For type $A_n$, the generator $E_i$ has entries $1$ at $F_i$ and $0$ otherwise. 
        It follows directly from the definition of the degree $d_P$ that \ref{Claim515_Expression1} of the claim is true.
        
    \subsubsection*{Type $B_n$}
        For type $B_n$, the generator ${E}_i$ has entries 1 at all diagonal pairs incident at $i$ or $\overline{i}$, where the entry at the diameter $[i,\overline{i}]$ is also $1$.
        By symmetry, for any diagonal pair $L=\{l,\overline{l}\}$, $(E_i)_L=1$ if and only if $l\in F_i$ or $\overline{l}\in F_i$.    
        Hence, the dot product $E_i\cdot d_P$ is equal to (the number of  
        $l\in \bQ(0)$ such that $l$ or $\overline{l}$ appears in  $F_i$)$-$(the number of $l\in \bQ(1)$ such that $l$ or $\overline{l}$ appears in  $F_i$), where each appearance is weighted by the power of the corresponding variable in $P$.
        
        If $[i,\overline{i}]$ is not in $\bQ$, all variables in $P$ that have nonzero entries in $E_i$ have multiplicity $1$.
        Moreover, for $l\neq [i,\overline{i}]$, if $l\in F_i$, then $\overline{l}\not\in F_i$.
        Therefore, the above interpretation of $E_i\cdot d_P$ can be calculated as a sum over the two quadrilaterals $\bQ$ and $\overline{\bQ}$, which implies \ref{Claim515_Expression2} of the claim.

        If $[i,\overline{i}]\in \bQ(0)$, $P$ must be a type 3 exchange, so $\bQ=\overline{\bQ}$ and all variables appearing in $P$ have degree 1.
        Therefore, \ref{Claim515_Expression1} of the claim is true.
        
        If $[i,\overline{i}]\in \bQ(1)$, $P$ must be a type 2 exchange, and $x_{[i,\overline{i}]}$ is the variable with exponent $2$ in $P$.
        In this case, the diameter $[i,\overline{i}]$ is where the two primitive quadrilaterals $\bQ$ and $\overline{\bQ}$ overlap, so it is in both $\bQ(1)$ and $\overline{\bQ(1)}$.
        This accounts for the degree $2$ in the exchange relation $P$.
	See Figure \ref{fig:5151}. 
	Hence, \ref{Claim515_Expression2} of the claim is true.
	
	\begin{figure}
	\begin{center}
		\includegraphics[height=4cm]{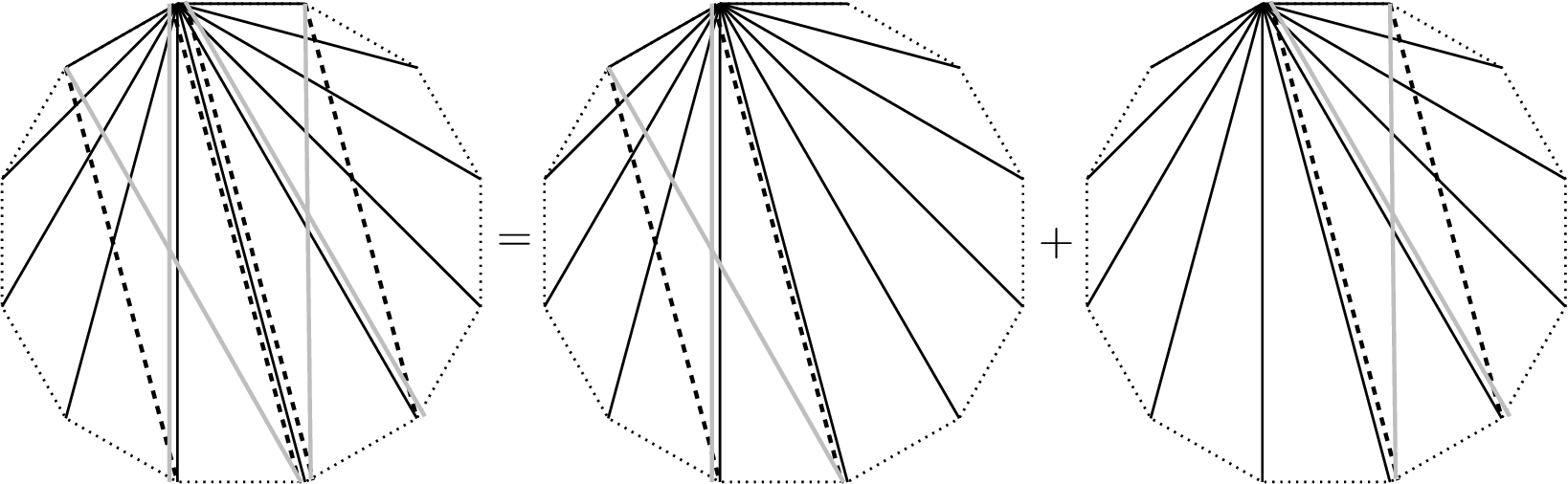}
		\caption{Type 2 exchange in type $B_n$ (Lemma \ref{Claim_512})}
		\label{fig:5151}
        \end{center}
\end{figure} 
        
    \subsubsection*{Type $C_n$}
        For type $C_n$, the generator ${E}_i$ has entries 1 at all diagonal pairs incident at $i$ or $\overline{i}$, except for the entry at the diameter $[i,\overline{i}]$, which is $2$.

        The entry corresponding to a set of diagonals $L$ in $E_i$ can be interpreted as the number of appearances of the vertices $i$ and $\overline{i}$ in each diagonal of $L$.
        With this interpretation, we only need to consider one diagonal in each diagonal pair.
        Hence, in the following calculations, we only consider one diagonal representative for each variable appearing in $P$.
        
        The dot product $E_i\cdot d_P$ is equal to (the number of appearances of $i$ and $\overline{i}$ in the elements of $\bQ(0)$)$-$(the same for $\bQ(1)$), where each appearance is weighted by the power of the corresponding variable in $P$.

        If $[i,\overline{i}]\not\in\bQ(0)$, all variables in $P$ that have nonzero entries in $E_i$ have multiplicity $1$ in $P$. Therefore, using the above interpretation of $E_i\cdot d$ and breaking down the sum by $i$ and $\overline{i}$, \ref{Claim515_Expression3} of the claim is true.
	See Figure \ref{fig:5152}.

        If $[i,\overline{i}]\in \bQ(0)$, $P$ must be a type 3 exchange, so $\bQ=\overline{\bQ}$ and the variable corresponding to the pair of diagonals in $\bQ(1)$ has entry $-2$ in $d_P$.
	The two diagonals in $\bQ(1)$ are incident to $i$ and $\overline{i}$ respectively, which is equivalent to counting one diagonal in the pair with multiplicity $2$. See Figure \ref{fig:5153}.
	Hence, \ref{Claim515_Expression3} of the claim is true.

	\subsubsection*{Type $D_n$}
        If $P$ is a type 1 or 3 exchange, the primitive exchange relation is of the same form as in the exchanges of these types in a cluster algebra of type $B_{n-1}$, and we reduce to the previously discussed case.
	If $P$ is a type 2 exchange, the primitive term in the exchange relation corresponds to two diameters, $l$ and $\widehat{l}$, that are the same line but of different colors, and their corresponding variables each have degree $1$ in $P$. 
        If the entry of $l$ is $1$ in ${E}_i$, the entry of $\widehat{l}$ is also $1$.
        Hence, the computation of ${E}_i\cdot d_P$ is equal to that in a type $B_{n-1}$ type 2 exchange, where the diameter $l$ has degree $2$ in the exchange relation but an entry $1$ in ${E}_i$.
    \end{proof}
	
	\begin{figure}
        \begin{center}
        \includegraphics[height=4cm]{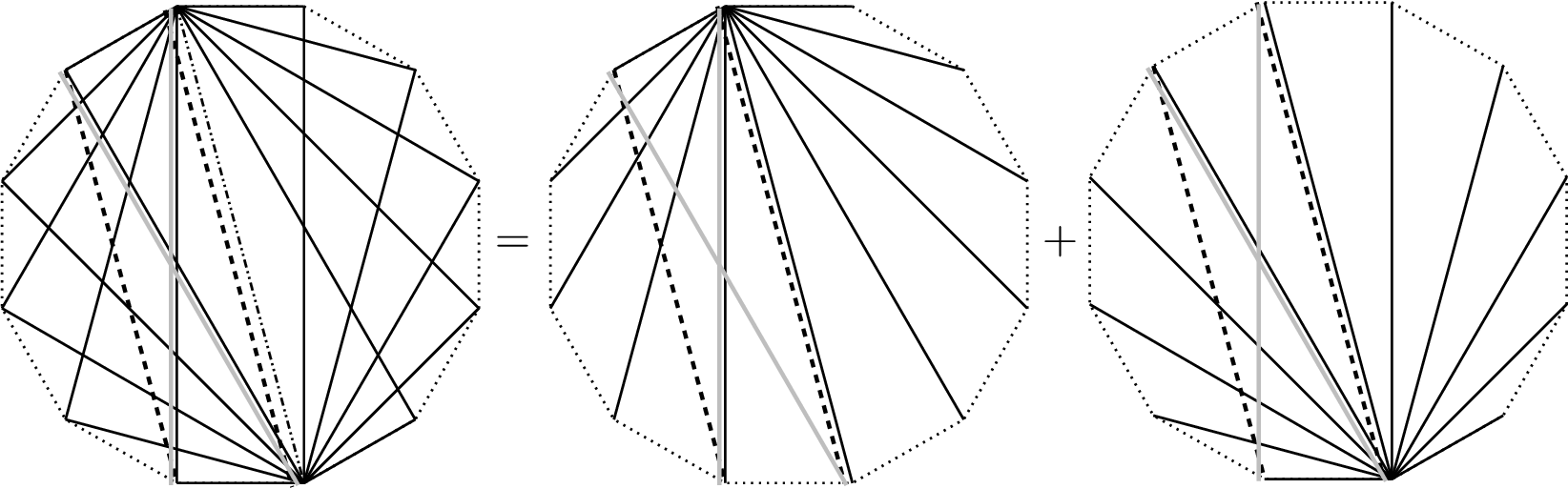}
	\caption{Type 1 or 2 exchange in type $C_n$ (Lemma \ref{Claim_512})}\label{fig:5152}
        \end{center}
\end{figure}
	\begin{figure}
	\begin{center}
        \includegraphics[height=4cm]{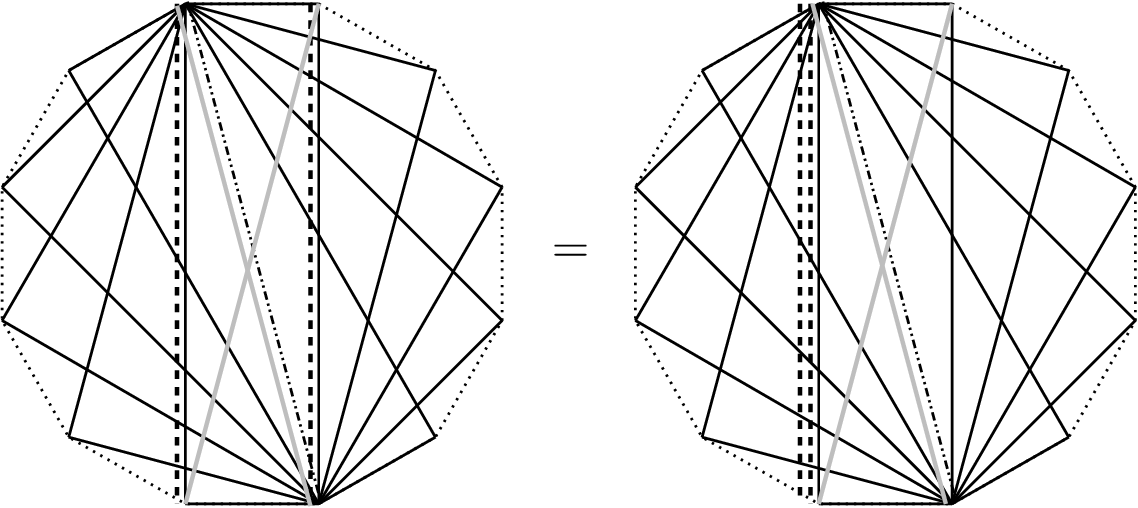}
	\caption{Type 3 exchange in type $C_n$ (Lemma \ref{Claim_512})}\label{fig:5153}
        \end{center}
\end{figure}

We now can complete the proof of Proposition \ref{Proposition_ABCn_LinSp}.

\begin{proof}[Proof of Proposition \ref{Proposition_ABCn_LinSp}]
    By Lemma \ref{Claim_512} and \ref{Lemma_FixingVertex_is_in_LinealitySpace}, $E_i\cdot d_P=0$ for any $i$ and any primitive exchange relation $P$.
    By \ref{Lemma_Of_The_Whole_Thesis}, this implies that $K(\Gamma)$ is contained in the lineality space of $\mcC_\mcA$.
   
    Note that $\#K$ is the number of frozen variables in $\mcA$.  By Lemma \ref{Lemma_Nullity_q}, this is exactly the dimension of the lineality space, so it remains to show that $K$ is a linearly independent set.

    Suppose $\sum_{i=1}^{q}a_i E_i=0$ for some $a_i\in \mathbb{R}$.
    Note that for any $i\neq j$, for any $E_k\in K$, the entry $(E_k)_{[i,j]}\neq 0$ if and only if $k=i$ or $k=j$,
    and $(E_i)_{[i,j]}=(E_j)_{[i,j]}$.
    Therefore, we have 
    \[
    a_i+a_j=0
    \]
    and so for any distinct $i,j,k$,
    \[
    a_i=-a_j=a_k=-a_i.
    \]
    Hence $a_i=0$ for any $i=1,\ldots, q$, so the set $K$ is linearly independent.
    It follows that $K$ is a basis of the lineality space.
\end{proof}

Likewise, we prove Proposition \ref{Proposition_Result3_Dn_Linsp}:

\begin{proof} [Proof of Proposition \ref{Proposition_Result3_Dn_Linsp}]
Let $P$ be an arbitrary primitive exchange relation.
 By Lemma \ref{Claim_512} and \ref{Lemma_FixingVertex_is_in_LinealitySpace}, $E_i\cdot d_P=0$ for any $i$.
Moreover, we claim that $u^\diam\cdot d_P=0$. By \ref{Lemma_Of_The_Whole_Thesis}, this will implies that $K(D_n)$ is contained in the lineality space of $\mcC_\mcA$.

We show that $u^{\mathrm{diam}}\cdot d_P=0$.
    It suffices to show that the number of variables corresponding to the red diameters and to the blue diameters appearing in the same term of $P$ is always the same. For type 1 exchanges, no diameters are involved. For type 2 exchanges, the only diameters appear in the primitive term, which contains two diameters of the same endpoints but different colors. For type 3 exchanges, the diameters are the two variables being exchanged, and they have different colors.
    In any case, $u^{\mathrm{diam}}\cdot d_P=0$.

We have now seen that the elements of $K(D_n)$ all belong to the lineality space. By Lemma \ref{Lemma_Nullity_q}, the dimension of the lineality space is one more than the number of frozen variables, which is also equal to the size of $K$. Hence, it suffices to show that $K(D_n)$ is a linearly independent set.
    Suppose 
    \[
    \sum_{i=1}^{n}a_i{E}_i+bu^{\mathrm{diam}}=0.
    \]
    Considering the entry indexed by the blue diagonal $[i,\overline{i}]$, we have
    \[
    a_i+b=0.
    \]
    Considering the entry indexed by the red diagonal $\widehat{[i,\overline{i}]}$, we have
    \[
    a_i-b=0.
    \]
    Since $i$ is arbitrary, $a_i=b=0$ for any $i$.
    Therefore, $K(D_n)$ is a basis of the lineality space.
\end{proof}

\subsection{Proofs of Theorems \ref{Theorem_Statement_Result3_ABCn} and \ref{Theorem_Statement_Result3_Dn}}\label{sec:proofs2}
We begin with a combinatorial lemma:
\begin{lemma}\label{Lemma_PrimQuadInCompleteGraph}
    Let $\bQ$ be a primitive quadrilateral in a regular polygon $\bP_N$. Consider an arc of $\bP_N$ with endpoints $i,j$, and let $F$ be the set of all diagonals and edges in $\bP_N$ whose endpoints lie on the arc.
    Then as long as $(\bQ(0)\cup \bQ(1))\cap F \neq\{[i,j]\}$, 
    \[
    \#(\bQ(0)\cap F)=\#(\bQ(1)\cap F).
    \]
\end{lemma}

\begin{proof}
  If zero or one of the vertices of $\bQ$ lie on the arc, the claim is clearly true, since no diagonals or edges of $\bQ$ can be in $F$.
  Since we suppose $(\bQ(0)\cup \bQ(1))\cap  F \neq\{[i,j]\}$, it is impossible to have exactly two vertices of $\bQ$ on the arc, unless $[i,j]$ is an edge, and doesn't belong to $\bQ(0)$ or $\bQ(1)$, in which case the claim is again clear.
    Hence, we may assume either 3 or 4 vertices of $\bQ$ are on the arc.

    In Figure \ref{Figure_LemmaQuad_Complete_Graph}, the cases where $3$ or $4$ vertices are on the arc are shown.
    Lines in $F$ are denoted by black solid lines, elements of $\bQ(0)$ by solid gray lines, and elements of $\bQ(1)$ by dashed black lines. The claim follows by inspection.
\end{proof}

    \begin{figure}
        \centering    
        \includegraphics[height=4cm]{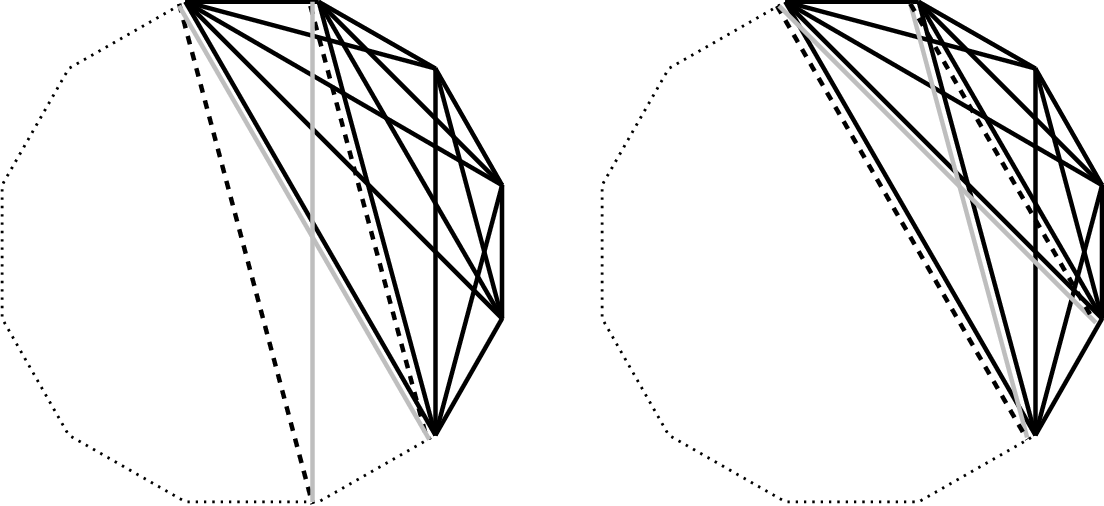}   
        \captionsetup{justification=centering}
	\caption{Cases in Lemma \ref{Lemma_PrimQuadInCompleteGraph}}
        \label{Figure_LemmaQuad_Complete_Graph}
    \end{figure}

    Let $P$ be a primitive exchange relation in one of the cluster algebras considered in 
Theorems \ref{Theorem_Statement_Result3_ABCn} and \ref{Theorem_Statement_Result3_Dn} with exchange quadrilateral(s) $\bQ$ (and $\overline\bQ$).
Define $L_P$ to be the set of the \emph{shortest} diagonals or edges in $\bP_N$ corresponding to variables in the primitive term of $P$.
Equivalently, $L_P$ is the set of shortest elements in $\bQ(1)$ (and $\overline{\bQ(1)}$). In type $A_n$, $L_P$ consists of a single element unless $n$ is odd and $\bQ=\overline \bQ$. In other types, $L_P$ always consists of two elements. The edges or diagonals in $L_P$ are always of non-maximal length by construction.

    For any non-maximal length diagonal or edge $l$ in $\bP_N$, the set of diagonals that share one endpoint with $l$ and have length one larger than $l$ contains exactly two elements. Moreover, they are $\bQ(0)$ for some primitive exchange quadrilateral $\bQ$.
    It directly follows that $l$ is (one of) the shortest diagonals or edges in $\bQ(1)$.
    Hence, except in type $D_n$, there exists a unique primitive exchange relation $P$ such that $l\in L_P$.
In type $D_n$, this is true as well except for when $l$ has length $n-1$, in which case there are two primitive exchange relations of type 3 with $l\in L_P$; they differ by swapping the colors on the diagonals.

\begin{lemma}\label{lemma:newdp}
Let $P$ be a primitive exchange relation and $l$ be any non-maximal length diagonal or edge.
Then 
\[
    \begin{cases}
    v(l)\cdot d_P >0 & \text{ if } l\in L_P,\\
    v(l)\cdot d_P =0 & \text{ otherwise.}
    \end{cases}
    \]
    \end{lemma}
    \begin{proof}
Let $\bQ$ (and $\overline{\bQ}$) be the exchange quadrilateral(s) of $P$. 
    We first show that $v(l)\cdot d_{P}>0$ if $l\in L_P$.
    Suppose $l\in L_P$.
    By definition, $L_P$ is the set of shortest diagonal(s) or edge(s) in $\bQ(1)$ (and $\overline{\bQ}(1)$).
    By definition of $v(l)$, $L_P$ contains the longest diagonal(s) or edge(s) that have nonzero entries in $v(l)$.
    Hence, the only common nonzero entry of $v(l)$ and $d_P$ is in $L_P$, and its entry is negative in both $v(l)$ and $d_{P}$.
    Therefore, $v(l)\cdot d_{P}>0$.

    Next, we show that $v(l)\cdot d_{P}=0$ for $l\not\in L_P$. 
    By definition, elements of $L_P$ are of non-maximal lengths. 
    By definition of $v(l)$, none of the nonzero entries in $v(l)$ is indexed by a diameter.
    Moreover, if $L$ is a longest non-diameter pair in type $B_n$, $C_n$, or $D_n$, $L$ has a nonzero entry in $v(l)$ if and only if $l\in L$.
    By the assumption $l\not\in L_P$, such $L$ must not be $L_{P}$, so the variable corresponding to $L$ cannot not appear in $P$ with exponent $2$.
    Therefore, in type $B_n$ or $C_n$, the different multiplicities of diameters and diagonals in the different exchange types do not affect the calculation.
    
    We thus  have 
    \[
	    v(l)\cdot d_{P}=\sum_{\bR\in\{\bQ,\overline \bQ\}}\#(\bR(0)\cap F_l)-\#(\bR(1)\cap F_l)
    \]
    where $F_l$ is as in the definition of $v(l)$.
    For $l\not\in L_P$, this vanishes by Lemma \ref{Lemma_PrimQuadInCompleteGraph}.
    \end{proof}

    \begin{proof}[Proof of Theorem \ref{Theorem_Statement_Result3_ABCn}]
	    By Lemma \ref{Lemma_Of_The_Whole_Thesis} it suffices to show that for every primitive exchange relation $P$, there exists an element $g\in\mcG$ such that \begin{equation}\label{eqn:g_P}
		    g\cdot d_P>0\qquad\textrm{and} \qquad g\cdot d_{P'}=0\ \textrm{for all primitive }P'\neq P,
    \end{equation}
    and every $g\in\mcG$ satisfies \eqref{eqn:g_P} for some primitive exchange relation $P$.
    This follows from Lemma \ref{lemma:newdp}, since every non-maximal length diagonal or edge belongs to $L_P$ for exactly one choice of $P$.
\end{proof}

For the proof in type $D_n$, we will need to pay special attention to primitive exchange relations $P$ of type $3$. Let the elements of $\bQ(0)$ be adjacent to vertices $i$ and $i+1$. We call the diagonal adjacent to $i$ the \emph{preceding diagonal}, and the diagonal adjacent to $i+1$ the \emph{succeeding diagonal}; we include the information of the colors of the diagonals as they appear in the relation $P$.

\begin{lemma}\label{lemma:dpnew2}
	Let $1\leq j\leq n$ and $0\leq k \leq n-1$. Let $P$ be any primitive exchange relation in type $D_n$ for $n\geq 4$.
	\begin{enumerate}
		\item
	Then $w(j,k)\cdot d_P=0$ unless $P$ is the type 3 exchange whose succeeding diagonal is $[j,\overline j]$ or whose preceding diagonal is $[j+k,\overline j+k]$. In these two cases, we have $w(j,k)\cdot d_P=1$.

\item	Similarly, $(w(j,k)-\widehat{w}(j+k+1))\cdot d_P=0$ unless $P$ is the type 3 exchange whose succeeding diagonal is $[j,\overline j]$ or ${[j+k+2,\overline j+k+2]}$. In these two cases, $(w(j,k)-\widehat{w}(j+k+1))\cdot d_P$ is non-zero with opposite signs unless $k=n-2$.
\end{enumerate}
\end{lemma}
\begin{proof}
	We first claim that for any primitive exchange relation $P$, $w(i)\cdot d_P=0$ unless $P$ is primitive of type $3$ with $[i,\overline i]$ in $\bQ(0)$. Indeed, if the elements of $L_P$ are of length less than $n-2$, this follows from Lemma \ref{lemma:newdp}, as $P$ does not involve any diameters.
	If the elements of $L_P$ have length $n-2$ but do not include $[i+1,\overline i-1]$, then Lemma \ref{lemma:newdp} still gives $v([i+1,\overline i -1])\cdot d_P=0$; since $P$ doesn't involve $[i,\overline i]$ the claim follows. If instead $L_P$ does include $[i+1,\overline i-1]$, then $v([i+1,\overline i -1])\cdot d_P=1$ but $e_{[i,\overline i]}\cdot d_P=1$, and the terms cancel to give zero. Finally, if $P$ is primitive of type $3$ then Lemma \ref{lemma:newdp} still gives $v([i+1,\overline i -1])\cdot d_P=0$, so the claim follows in this case as well.

	Similarly, $\widehat w(i)\cdot d_P=0$ unless $P$ is primitive of type $3$ with $\widehat{[i,\overline i]}$ among its diagonals.
Analogous arguments show that if $P$ is primitive of type $3$ and $[i,\overline i]$ (respectively $\widehat{[i,\overline i]}$) is among its diagonals, $w(i)\cdot d_P=1$ (respectively $\widehat w(i)\cdot d_P=1$).
Finally, we note that Lemma \ref{lemma:newdp} applied to $l=[j+k,\overline j+k+1]$ yields that $v(l)\cdot d_P\neq 0$ if and only if $P$ is the exchange relations with either $[j+k,\overline j+k]$ or $\widehat{[j+k,\overline j+k]}$ as its preceding diagonal, in which case the dot product is $1$.

	The claims of the lemma now follow from the above, and the definition of $w(j,k)$. 
\end{proof}

\begin{proof}[Proof of Theorem \ref{Theorem_Statement_Result3_Dn}]
	Let $\mcG$ be the set of $v(l)$ and $w(j,k)$ in the statement of the theorem.
Since the dimension of the lineality space of $\mcC_\mcA$ is one more than the number of frozen variables by Lemma \ref{Lemma_Nullity_q}, the cone $\mcC_\prim$ has dimension one less than the number $p$ of cluster variables. This means that $\mcG$ contains the generators for $\mcC_\prim$ if for every two primitive exchange relations $P_1,P_2$, one of the following occurs:
\begin{enumerate}
	\item There exists $g\in \mcG$ such that for any primitive exchange relation $P$,
	$g\cdot d_P\geq 0$ with equality if $P\neq P_1,P_2$, and strict inequality for at least one of $P_1,P_2$; or\label{innernormal}
	\item There exists $g'\in \RR^{\mcV\cup\mcW}$ such that for any primitive exchange relation $P\neq P_1,P_2$
		$g'\cdot d_P= 0$, and $g'\cdot P_1,g'\cdot P_2$ are non-zero with opposite signs.\label{notinnernormal}

\end{enumerate}
Indeed, in the first case the $d_P$ with $P\neq P_1,P_2$ generate a face of $\mcC_\prim$ and $g$ is an inward normal vector to this face, and in the second case the $d_P$ with $P\neq P_1,P_2$ do not generate a face of $\mcC_\prim$. Thus, we see that $\mcG$ contains an inward normal vector to each facet of $\mcC_\prim$, hence contains generators for $\mcC_\mcA$ by Lemma \ref{Lemma_Of_The_Whole_Thesis}.

For any primitive exchange relation $P$ not of type 3, Lemma \ref{lemma:newdp} implies that if $P$ is one of $P_1$ or $P_2$, we may take $g=v(l)$ for $l\in L_P$ to satisfy \eqref{innernormal} above. We may thus assume that both $P_1$ and $P_2$ are primitive of type 3.
If the succeeding diagonals are of different colors, we way assume that the blue diagonal of $P_1$ is succeeding and of the form $[j,\overline j]$ 
for $1\leq j \leq n$, and the blue diagonal of $P_2$ is preceding and of the form $[j+k,\overline j+k]$ for $0\leq k \leq n-1$.
Lemma \ref{lemma:dpnew2} implies that we may take $g=w(j,k)$ to satisfy \eqref{innernormal} above.

Suppose instead that succeeding diagonals are both blue. We may assume that they are of the form $[j,\overline j]$ 
for $1\leq j \leq n$ and $[j+k+2,\overline j+k+2]$ and $0\leq k \leq n-1$, $k\neq n-2$. 
Then Lemma \ref{lemma:dpnew2} implies that we may take $g'=w(j,k)-\widehat{w}(j+k+1)$ to satisfy \eqref{notinnernormal}. 
If both succeeding diagonals are red, we may swap the roles of red and blue to obtain a similar conclusion.
We conclude that the generators of $\mcC_\mcA$ are among the elements of $\mcG$.

We next note that for every $g\in \mcG$ and every primitive relation, $g\cdot d_P\geq 0$ by Lemma \ref{lemma:newdp} and \ref{lemma:dpnew2}. Hence, $\mcG$ is contained in $\mcC_\mcA$. 

Finally, fix $g\in \mcG$. We will show that $g$ in fact generates a ray of $\mcC_\mcA$. For this, we will show that there exists an element $u\in \mcC_\prim$ that has dot product $0$ with $g$, and positive dot product with all other elements of $\mcG$. 
Indeed, if $g=v(l)$ for $l$ of length at most $n-2$, there is a unique exchange relation $P$ with $l\in L_P$. Taking $u=\sum_{P'\neq P} d_{P'}$, Lemmas \ref{lemma:newdp} and \ref{lemma:dpnew2} imply that $u$ has the desired property.
If we instead take $g=w(j,k)$, let $P_1$ and $P_2$ respectively be the unique exchange relations with $[j,\overline j]$ and $[j+k,\overline j+k]$ respectively the succeeding and preceding diagonals of $P_1$ and $P_2$. We then take $u=\sum_{P'\neq P_1,P_2} d_{P'}$. That $u$ has the desired property follows again from Lemmas \ref{lemma:newdp} and \ref{lemma:dpnew2}.
\end{proof}

\section{Classical Types without Frozen Variables}\label{sec:nofrozen}
\subsection{Results}\label{sec:cfresults}
In this section, we describe generators of the Gr\"obner cones $\mcC_\mcA$ when $\mcA$ is of type $A_n$, $B_n$, $C_n$, or $D_n$ and has no frozen variables.
As in the previous section, we let $N$ be $n+3$, $2n+2$, $2n+2$, or $2n$ respectively and will describe elements of $\mcA$ using (pairs of) diagonals of $\bP_N$.
Let $\mcA'$ be the cluster algebra of the same type as $\mcA$ with frozen variables as in \S\ref{sec:comb}. We may obtain $\mcC_\mcA$ by intersecting the cone $\mcC_{\mcA'}$ with the linear subspace $\Lambda$ where coordinates indexed by frozen variables of $\mcA'$ vanish, and projecting the result onto the coordinates indexed by cluster variables.
By forgetting the projection, we will describe $\mcC_\mcA$ below as living in the same space $\RR^{\mcV\cup \mcW}$ as $\mcA'$, where $\mcV$ are the cluster variables of $\mcA$ (or $\mcA'$), and $\mcW$ are the frozen variables of $\mcA'$.

We say that a diagonal or an edge $l$ of $\bP_N$ is \emph{even} (respectively \emph{odd}) if its length is even (respectively odd).
For a diagonal or edge $l$ that is not a diameter, we let $\delta_-(l),\delta_+(l)$ be respectively the first and last vertices on the minor arc of $l$, ordered in a counterclockwise fashion.

In \S\ref{sec:linspace} we defined for any vertex $i$ of $\bP_N$ a vector $E_i$, depending on the type of the cluster algebra. We now define
\[
E=\sum_{i=1}^N E_i.
\]

For any diagonal or edge $l$ that is not a diameter, we set
\begin{align*}
	E(l)&=\displaystyle\sum_{k=1}^{\lfloor \len(l)/2 \rfloor} E_{\delta_-(l)+2k-1}\\
	\overline{E}(l)&=\displaystyle\sum_{k=1}^{\lfloor (N-\len(l))/2 \rfloor} E_{\delta_+(l)+2k-1}\\
	H(l)&=\sum_{k=0}^{\len(l)-1} (-1)^k E_{\delta_-(l)+k}.
\end{align*}
Using this, we then define
\begin{align*}
	\widetilde {v}(l)=\begin{cases}
	v(l)+\frac{1}{2}E-\overline{E}(l)\qquad& \textrm{$l$ odd in type $A_n$ and $n$ even}\\
	v(l)+\frac{1}{4}E-E([\delta_+(l),\overline{\delta_-(l)}])\qquad& \textrm{$l$ odd in type $B_n$,$C_n$,$D_n$ and $N/2$ odd}\\
			v(l)+E(l)\qquad&\textrm{otherwise}
\end{cases}.
\end{align*}

\begin{theorem}\label{thm:cf1}
	Suppose that $n$ is even and $\mcA$ is of type $A_n$, $B_n$, or $C_n$ with no frozen variables. Then $\mcC_{\mcA}$ is a pointed simplicial cone with rays generated by $\widetilde{v}(l)$
as $l$ ranges over all diagonals and edges of $\bP_N$ of non-maximal length.
\end{theorem}

\begin{theorem}\label{thm:cf2}
	Suppose that $n$ is odd and $\mcA$ is of type $A_n$ ($n\geq 1$), or $B_n$ or $C_n$ ($n\geq 3$) with no frozen variables. The lineality space of $\mcC_\mcA$ is the span of the vector $\sum_{i=1}^{N} (-1)^iE_{i}$. The  rays of $\mcC_\mcA$  modulo lineality space are generated by
$\widetilde{v}(l)$
for $l$ an even diagonal of non-maximal length of $\bP_N$,
and
\begin{align*}
	\widetilde{v}(l)+\widetilde{v}(l')+H([\delta_+(l),\delta_+(l')])
\end{align*}
for $l,l'$ odd diagonals or edges of non-maximal length satisfying $\delta_+(l)+\delta_+(l')\equiv 1\mod 2$.
\end{theorem}
As always, for $D_n$ we need slightly more notation. Define $\widetilde{w}(j,k)$ exactly as we defined $w(j,k)$ in \S\ref{sec:rays}, except we replace every occurrence of $v(l)$ in the formulas of $w(i),\widehat{w}(i),w(j,k)$ with $\widetilde{v}(l)$.
\begin{theorem}\label{thm:cf3}
	Suppose that $n\geq 5$ is odd and $\mcA$ is of type $D_n$ with no frozen variables. The lineality space of $\mcC_\mcA$ is the span of the vector $u^{\mathrm{diam}}$. The  rays of $\mcC_\mcA$ modulo lineality space are generated by $\widetilde{v}(l)$  as $l$ ranges over all diagonals and edges of $\bP_N$ of length at most $n-2$, and by $\widetilde{w}(j,k)$ for $1\leq j\leq n$, $0\leq k\leq n-1$.
\end{theorem}

\begin{theorem}\label{thm:cf4}
	Suppose that $n\geq 4$ is even and $\mcA$ is of type $D_n$ with no frozen variables. The lineality space of $\mcC_\mcA$ is the span of the vectors $\sum_{i=1}^{N} (-1)^i E_{i}$ and $u^{\mathrm{diam}}$. The  rays of $\mcC_\mcA$ modulo lineality space are generated by:
	\begin{enumerate}
		\item 	$\widetilde{v}(l)$  as $l$ ranges over all even diagonals of $\bP_N$ of length at most $n-2$;
		\item  $\widetilde{w}(j,k)$ for $1\leq j\leq n$, $0\leq k\leq n-1$, and $k$ even;
		\item $\widetilde{v}(l)+\widetilde{v}(l')+H([\delta_+(l),\delta_+(l')])$ for $l,l'$ odd diagonals or edges of non-maximal length satisfying $\delta_+(l)+\delta_+(l')\equiv 1\mod 2$;
		\item 	$\widetilde{w}(j,k)+\widetilde{v}(l)+H([j+k,\delta_+(l)])$ for 
$1\leq j\leq n$, $0\leq k\leq n-1$, $k$ odd, and $l$ an odd diagonals or edge of non-maximal length satisfying $j+\delta_+(l)\equiv 0 \mod 2$.
	\end{enumerate}
		\end{theorem}

		\subsection{Proof of Theorems \ref{thm:cf1}, \ref{thm:cf2}, \ref{thm:cf3}, and \ref{thm:cf4}}
		We continue with notation as in \S\ref{sec:cfresults}.

		\begin{lemma}\label{lemma:inlambda1}
			Let $l$ be a diagonal or edge of $\bP_N$. If $l$ is odd, and either $n$ is odd and we are in types $A_n$, $B_n$, or $C_n$, or $n$ is even and we are in type $D_n$, then the projection of $\widetilde{v}(l)$ to $\RR^\mcW$ is the product of $-1$ with the basis vector corresponding to the edge $[\delta_+(l)-1,\delta_+(l)]$. In all other cases, $\widetilde{v}(l)\in \Lambda$.
		\end{lemma}
		\begin{proof}
			Suppose first that we are in one of the cases where $\widetilde{v}(l)=v(l)+E(l)$. The projection of $v(l)$ to $\RR^\mcW$ is supported on the edges of $\bP_N$ on the minor arc of $l$ (with each entry $-1$). The projection of $E(l)$ to $\RR^\mcW$ is supported on the edges of the minor arc of $[\delta_-(l),\delta_-(l)+2\lfloor \len(l)/2\rfloor]$ (with each entry $1$). If $\len(l)$ is even, these two arcs agree and the terms cancel. If $\len(l)$ is odd, then the edge $[\delta_+(l)-1,\delta_+(l)]$ in the minor arc of $l$ does not appear in the latter arc. The claim follows in these cases.

			Suppose instead that $l$ is odd and we are in type $A_n$ with $n$ even. Then the projection of $v(l)+\frac{1}{2}E$ to $\RR^\mcW$ is supported on the edges of the major arc of $l$ (with entries $1)$, and $\overline{E}(l)$ is also supported on the edges of the major arc of $l$ (with entries $1)$. The claim follows in this case.

			If finally $l$ is odd and we are in type $B_n,C_n,D_n$ with $N/2$ odd, the projection of $v(l)+\frac{1}{4}E$ to $\RR^\mcW$ is supported on the edges of the minor arc of $[\delta_+(l),\overline \delta_-(l)]$ (with entries $1)$ as is the projection of $E([\delta_+(l),\overline \delta_-(l)])$ (with entries $1$). The claim follows in this case as well.
		\end{proof}

		\begin{lemma}\label{lemma:F}
			Let $l$ be an odd diagonal or edge of $\bP_N$. The projection of $H(l)$ to $\RR^\mcW$ is the basis vector corresponding to the edge $[\delta_-(l)-1,\delta_-(l)]$ plus the basis vector corresponding to the edge $[\delta_+(l)-1,\delta_+(l)]$.
		\end{lemma}
		\begin{proof}
When projecting to $\RR^\mcW$, the $k$th term in the alternating sum defining $H(l)$ is supported at exactly two edges, namely 
$[\delta_-(l)+k-1,\delta_-(l)+k]$ and $[\delta_-(l)+k,\delta_-(l)+k+1]$.
Since the sum is alternating, everything cancels except the $[\delta_-(l)+k-1,\delta_-(l)+k]$-coordinate when $k=0$ and the $[\delta_-(l)+k,\delta_-(l)+k+1]$-coordinate when $k=\len(l)-1$, and the claim follows.
				\end{proof}

		\begin{lemma}\label{lemma:inlambda2}
			Suppose we are in type $D_n$ for $n\geq 4$ and fix $1\leq j \leq n$ and $0\leq k \leq n-1$.  If $n$ is even and $k$ is odd, the projection of $\widetilde w(j,k)$ to $\RR^\mcW$ is the product of $-1$ with the basis vector corresponding to the edge $[j+k-1,j+k]$. In all other cases, $\widetilde{w}(j,k)\in\Lambda$.
		\end{lemma}
		\begin{proof}
			The only potential contributions to the projection of $\widetilde w(j,k)$ to $\RR^\mcW$ come from the projections of the $\widetilde{v}(l)$ appearing in the formula for $\widetilde w(j,k)$. When $n$ is odd, all $\widetilde{v}(l) \in \Lambda$ by Lemma \ref{lemma:inlambda1} so $\widetilde{w}(j,k)\in\Lambda$.

			When $n$ is even, we similarly only need to worry about contributions from $\widetilde{v}(l)$ when $l$ is odd. The diagonals $l$ with $\widetilde{v}(l)$ appearing in $w(j,k)$ for $k$ even all have length $n-2$, hence are even, and so again $\widetilde{w}(j,k)\in\Lambda$.
			For $k$ odd, the only contribution with an odd diagonal is exactly $\widetilde{v}([j+k,\overline j+k+1])$ (the diagonal has length $n-1$). The claim now follows from Lemma \ref{lemma:inlambda1}.
		\end{proof}

		\begin{lemma}\label{lemma:incone}

The elements listed as  ray generators of $\mcC_\mcA$ in Theorems  \ref{thm:cf1}, \ref{thm:cf2}, \ref{thm:cf3}, or \ref{thm:cf4} belong to $\mcC_\mcA$.
		\end{lemma}
		\begin{proof}
			Let $g$ be one of the elements listed as a generator.
			By Theorems \ref{Theorem_Statement_Result3_ABCn} and \ref{Theorem_Statement_Result3_Dn}, $g$ is a non-negative linear combination of ray generators of $\mcA'$ with elements from the lineality space of $\mcA'$; as such, it belongs to $\mcC_{\mcA'}$. 

			Consider 
			\[
g=\widetilde{v}(l)+\widetilde{v}(l')+H([\delta_+(l),\delta_+(l')])
			\]
			as in Theorem \ref{thm:cf2} or Theorem \ref{thm:cf4}. By Lemma \ref{lemma:inlambda1} and Lemma \ref{lemma:F}, projecting to $\RR^\mcW$ we obtain $0$.
Likewise, for $\widetilde{w}(j,k)+\widetilde{v}(l)+H([j+k,\delta_+(l)])$ as in Theorem \ref{thm:cf4}, projection to $\RR^\mcW$ we obtain $0$ by Lemmas \ref{lemma:inlambda1}, Lemma \ref{lemma:F}, and \ref{lemma:inlambda2}.

For all other types of generators $g$, we see directly from Lemma \ref{lemma:inlambda1} or \ref{lemma:inlambda2} that $g\in \Lambda$. In any case, we conclude that $g$ always belongs to $\Lambda\cap \mcC_{\mcA'}$, which is exactly $\mcC_\mcA$.
		\end{proof}
		\begin{lemma}\label{lemma:ws}
Suppose we are in type $D_n$ with $n\geq 4$ even. Let $1\leq j,j'\leq n$, $0\leq k,k'\leq n-1$.
		Let $r,r'$ respectively be the remainders of $j'+k'-j$ and $j+k-j'$ by $n$.
Then
\[u:={w}(j,k)+{w}(j',k')-{w}(j,r)-{w}(j',r')
\]
belongs to the lineality space of $\mcC_{\mcA'}$.
\end{lemma}
\begin{proof}
For all primitive exchange relations $P$, we claim that $u\cdot d_P=0$; the claim will follow.
By Lemma \ref{lemma:dpnew2}, there are (at most) four primitive exchange relations $P$ such that the dot product of $d_P$ with one of the four terms of $u$ is non-zero: the relations with  $[j,\overline j]$ or  $[j',\overline j']$ the succeeding diagonal, and the relations with $[j+k,\overline j+k]=[j'+r',\overline j'+r']$ or $[j'+k',\overline j'+k']=[j+r,\overline j+r]$ the preceding diagonal. Each of these four $d_P$ has dot product $1$ with one of the positive summands of $u$, and dot product $1$ with one of the negative summands of $u$, cancelling out.
\end{proof}

		\begin{lemma}\label{lemma:gen}
			Any element of $\mcC_\mcA$ can be written as the sum of an element of the lineality space of $\mcC_\mcA$ plus a non-negative linear combination of the generators listed in Theorem \ref{thm:cf1}, \ref{thm:cf2}, \ref{thm:cf3}, or \ref{thm:cf4}.
		\end{lemma}
		\begin{proof}
			Fix the type of $\mcA$, and let $\mcG$ be the set of generators of $\mcC_{\mcA'}$ from either Theorem \ref{Theorem_Statement_Result3_ABCn} and \ref{Theorem_Statement_Result3_Dn}. Then an arbitrary element $u\in \mcC_\mcA$ can be written as a 
\[
	u=u_0+\sum_{g\in \mcG} \lambda_g\cdot g
\]
where $u_0$ is in the lineality space of $\mcC_{\mcA'}$ and $\lambda_g\geq 0$. For each generator $g$, let $\widetilde{g}$ be the vector obtained by replacing any $v(l)$ with $\widetilde{v}(l)$; these differ only by elements of the lineality space so we may still write
\[
	u=\widetilde{u}_0+\sum_{g\in \mcG} \lambda_g\cdot \widetilde{g}
\]
for some $\widetilde{u}_0$ in the lineality space of $\mcC_{\mcA'}$. In the cases of Theorem  \ref{thm:cf1} and \ref{thm:cf3}, all the $\widetilde{g}$ belong to $\Lambda$, so $\widetilde{u}_0$ does as well, so $\widetilde{u}_0$ is in the lineality space of $\mcC_\mcA$ and the claim of the lemma follows.

For the cases of Theorems  \ref{thm:cf2} and \ref{thm:cf4}, $N$ is even (in type $A_n$) and divisible by $4$ (in types $B_n$,$C_n$,$D_n$). Consider the projection $\pi:\RR^{\mcV\cup\mcW}\to \RR$ sending the basis vector corresponding to $[i,i+1]$ to $1$ if $i$ is odd and $-1$ if $i$ is even, and everything else to zero. This is well-defined by the above condition on $N$.

By examining the generators of the lineality space of $\mcC_\mcA'$ (Theorems \ref{Proposition_ABCn_LinSp} and \ref{Proposition_Result3_Dn_Linsp}) we see that all elements of the lineality space project to $0$. On the other hand, by Lemmas \ref{lemma:inlambda1} and \ref{lemma:inlambda2}, $\pi(\widetilde{g})=0$ if and only if $\widetilde{g}\in \Lambda$; denote the set of these $g$ by $\mcG_0$. For all other $g\in\mcG$, we have $\pi(\widetilde{g})=\pm 1$. Let 
$\mcG_+$ be those $g$ with $\pi(\widetilde{g})= 1$ and $\mcG_-$ be those $g$ with $\pi(\widetilde{g})=- 1$.

Since $\pi(u)=0$, it follows that 
\[
\sum_{g\in\mcG_+} \lambda_g=\sum_{g\in\mcG_-} \lambda_g.
\]
This means that we may find $\mu_{g,g'}\in \RR_{\geq 0}$ for $(g,g')\in \mcG_+\times \mcG_-$ such that 
\[
	u=\widetilde{u}_0+\sum_{g\in \mcG_0} \lambda_g\cdot \widetilde{g}+\sum_{(g,g')\in \mcG_+\times \mcG_-} \mu_{g,g'}\cdot (\widetilde{g}+\widetilde{g'}).
\]

If we are in type $D_n$ for $n$ even and $(g,g')=(w(j,k),w(j',k'))$ with $k,k'$ odd, by Lemma \ref{lemma:ws} we may replace $\widetilde{g}+\widetilde{g'}$ with $\widetilde{w}(j,r)+\widetilde{w}(j',r')$, where $r,r'$ are as in the statement of the lemma. Moreover, since $r,r'$ are even, $w(j,r),w(j',r')\in\mcG_0$.

By Lemmas \ref{lemma:inlambda1} and \ref{lemma:inlambda2} it follows that for any $(g,g')\in \mcG_+\times \mcG_-$ not of that form, $\widetilde{g}+\widetilde{g'}$  
differs from one of the listed generators of Theorems  \ref{thm:cf2} and \ref{thm:cf4} by an element of the lineality space of $\mcC_{\mcA'}$.
Combining this with Lemma \ref{lemma:incone} the claim of the lemma follows.
		\end{proof}
We now complete the proof of the theorems:
\begin{proof}[Proof of Theorems \ref{thm:cf1}, \ref{thm:cf2}, \ref{thm:cf3}, and \ref{thm:cf4}]
	To understand the lineality space of $\mcC_\mcA$, we use Theorems \ref{Proposition_ABCn_LinSp} and \ref{Proposition_Result3_Dn_Linsp}.
	It is straightforward to show that the projections of the vectors $E_i$ to $\RR^\mcW$ are linearly independent if $n$ is even and $\mcA$ is of type $A_n$, $B_n$, or $C_n$, or $n$ is odd and $\mcA$ is of type $D_n$. This implies that in these cases the intersection of the lineality space of $\mcC_{\mcA'}$ with $\Lambda$ is $0$ (in types $A_n$, $B_n$, or $C_n$) or generated by $u^\diam$ (in type $D_n$).

	On the other hand, if $n$ is odd and $\mcA$ is of type $A_n$, $B_n$, or $C_n$, or $n$ is even and $\mcA$ is of type $D_n$, the projections of the vectors $E_i$ to $\RR^\mcW$ are not linearly independent. However, any set of projections with all but one of the $E_i$ is linearly independent, so we conclude that the lineality space has dimension $1$ (in types $A_n$, $B_n$, or $C_n$) or $2$ (in type $D_n$, again because of the existence of $u^\diam$). In these cases, $\sum_{i=1}^N(-1)^i E_i$ (and $u^\diam$) clearly belong to the intersection of the lineality space of $\mcC_{\mcA'}$ and $\Lambda$. The claims regarding the lineality spaces of $\mcC_\mcA$ follow.

Let $\widetilde \mcG$ be the set of generators listed in the statement of one of the theorems.
	We have already seen in Lemmas \ref{lemma:incone} and \ref{lemma:gen} that $\mcC_\mcA$ is generated by its lineality space and $\widetilde \mcG$. It remains to show that any element $\widetilde g\in \widetilde \mcG$ generates a ray of $\mcC_\mcA$ (modulo lineality space).
If $\widetilde g$ is a translate of a ray generator $g$ of $\mcC_{\mcA'}$ by an element of its lineality space, it is immediate that $\widetilde g$ also generates a ray. We thus may consider generators $\widetilde g$ arising as the translate of $g_1+g_2$, where $g_1,g_2$ are ray generators for $\mcC_{\mcA'}$.

In these cases, take $u=\sum_{P\ :\ \widetilde g\cdot d_P=0} d_P$, where the sum is over all primitive exchange relations $P$ such that $\widetilde g\cdot d_P=0$. 
We will show that $\widetilde g'\cdot u>0$ for all other $\widetilde g'\in \widetilde \mcG$. Since $ \widetilde g\cdot u=0$ this implies that $\widetilde{g}$ generates a ray.

If $g_1=v(l)$, $g_2=v(l')$, then there are only two $P$ not appearing in the sum defining $u$, namely those $P$ with $l\in L_P$ or $l'\in L_P$. It follows from Lemmas \ref{lemma:newdp} and \ref{lemma:dpnew2} that $u$ has the desired property, since $\widetilde \mcG$ does not include a translate of $v(l)$ or $v(l')$. If $g_1=v(l)$, $g_2=w(j,k)$, there are three exchange relations not appearing in the sum defining $u$, namely, the relation $P$ with $l\in L_P$, the relation with $[j,\overline j]$ the succeeding diagonal, and the relation with $[j+k,\overline j+k]$ the preceding diagonal. Again Lemmas \ref{lemma:newdp} and \ref{lemma:dpnew2} imply that $u$ has the desired property, since $\widetilde \mcG$ does not include a translate of $v(l)$ or $w(j,k)$.
		\end{proof}

\bibliographystyle{alpha}
\bibliography{references} 	

\end{document}